\documentclass{amsart}

\usepackage{latexsym}
\usepackage{amssymb, latexsym}
\usepackage{amsmath}
\usepackage{mathrsfs}
\usepackage{amsxtra}
\usepackage{amsthm}
\usepackage{verbatim}
\newtheorem{theorem}{Theorem}[section]
\newtheorem{corollary}[theorem]{Corollary}

\newtheorem{lemma}[theorem]{Lemma}

\newtheorem*{main}{Main~Theorem}

\theoremstyle{definition}
\newtheorem{definition}[theorem]{Definition}

\newtheorem{remark}[theorem]{Remark}

\newcommand{\image}{\mathbin{\hbox{\tt\char'42}}}
\newcommand{\Erdos}{Erd\H os}
\newcommand{\Los}{\L o\'s}
\newcommand{\Mbar}{{\overline{M}}}
\newcommand{\Nbar}{{\overline{N}}}
\newcommand{\Ult}{\text{Ult}}
\newcommand{\Union}{\bigcup}

\newcommand{\of}{\subseteq}
\newcommand{\lt}{\smalllt}
\newcommand{\lesseq}[1]{{\smallleq}#1}
\newcommand{\smallleq}{\mathrel{\mathchoice{\raise2pt\hbox{$\scriptstyle\leq$}}{\raise1pt\hbox{$\scriptstyle\leq$}}{\raise1pt\hbox{$\scriptscriptstyle\leq$}}{\scriptscriptstyle\leq}}}
\newcommand{\smalllt}{\mathrel{\mathchoice{\raise2pt\hbox{$\scriptstyle<$}}{\raise1pt\hbox{$\scriptstyle<$}}{\raise0pt\hbox{$\scriptscriptstyle<$}}{\scriptscriptstyle<}}}
\newcommand{\smallgt}{\mathrel{\mathchoice{\raise2pt\hbox{$\scriptstyle>$}}{\raise1pt\hbox{$\scriptstyle>$}}{\raise0pt\hbox{$\scriptscriptstyle>$}}{\scriptscriptstyle>}}}

\newcommand{\smallgeq}{\mathrel{\mathchoice{\raise2pt\hbox{$\scriptstyle\geq$}}{\raise1pt\hbox{$\scriptstyle\geq$}}{\raise1pt\hbox{$\scriptscriptstyle\geq$}}{\scriptscriptstyle\geq}}}

\newcommand{\Add}{\mathop{\rm Add}}
\newcommand{\ltkappa}{{{\smalllt}\kappa}}
\newcommand{\leqkappa}{{{\smallleq}\kappa}}
\newcommand{\GCH}{{\rm GCH}}
\newcommand{\ORD}{\mathop{{\rm ORD}}}
\newcommand{\ZFC}{{\rm ZFC}}

\newcommand{\Godel}{G\"{o}del}
\newcommand{\Konig}{K\"{o}nig}
\newcommand{\one}{\mathop{1\hskip-2.5pt {\rm l}}}

\newcommand{\p}{\mathbb{P}}
\newcommand{\q}{\mathbb{Q}}

\newcommand{\e}{\mathbb{E}}
\newcommand{\rr}{\mathbb{R}}

\newcommand{\la}{\langle}
\newcommand{\ra}{\rangle}

\newcommand{\her}[1]{H_{{#1}^+}}

\newcommand{\tail}{\text{tail}}

\newcommand{\forces}{\Vdash}
\newcommand{\restrict}{\upharpoonright}

\numberwithin{equation}{section}

\renewcommand{\P}{\mathbb{P}}
\newcommand{\Q}{\mathbb{Q}}
\newcommand{\E}{\mathbb{E}}

\newcommand{\REG}{\rm REG}

\newcommand{\cf}{\mathop{\rm cf}}
\newcommand{\dom}{\mathop{\rm dom}}

\renewcommand{\tilde}{\widetilde}

\usepackage{hyperref}
\hypersetup{colorlinks=true, linkcolor=blue, citecolor=blue}

\begin{document}

\author[Brent Cody]{Brent Cody}
\address[Brent Cody]{
The Fields Institute for Research in Mathematical Sciences,
222 College Street,
Toronto, Ontario M5S 2N2,
Canada}
\email[B. ~Cody]{bcody@fields.utoronto.ca}
\urladdr{http://www.fields.utoronto.ca/~bcody/}

\author[Victoria Gitman]{Victoria Gitman}
\address[Victoria Gitman]{New York City College of Technology (CUNY), 300 Jay Street,
Brooklyn, NY 11201 USA} \email{vgitman@nylogic.org}

\title[Easton's thm. for Ramsey and strongly Ramsey cardinals.]{Easton's theorem for Ramsey and strongly Ramsey cardinals}
\today

\begin{abstract}
We show that, assuming \GCH, if $\kappa$ is a Ramsey or a strongly Ramsey cardinal and $F$ is a class function on the regular cardinals having a closure point at $\kappa$ and obeying the constraints of Easton's theorem, namely, $F(\alpha)\leq F(\beta)$ for $\alpha\leq\beta$ and $\alpha<\cf(F(\alpha))$, then there is a cofinality preserving forcing extension in which $\kappa$ remains Ramsey or strongly Ramsey respectively and $2^\delta=F(\delta)$ for every regular cardinal $\delta$.


\end{abstract}

\maketitle

\section{Introduction}
Since the earliest days of set theory, when Cantor put forth the Continuum Hypothesis in 1877, set theorists have been trying to understand the properties of the continuum function dictating the sizes of powersets of cardinals. In 1904, \Konig\ presented his false proof that the continuum is not an aleph, from which Zermelo derived the primary constraint on the continuum function, the Zermelo-\Konig\ inequality, that $\alpha<\cf(2^\alpha)$ for any cardinal $\alpha$. In the following years, Jourdain and Housedorff introduced the Generalized Continuum Hypothesis, and in another two decades \Godel\ showed the consistency of \GCH\ by demonstrating that it held in his constructible universe $L$.\footnote{For a full account of the early history of the \GCH\ see~\cite{moore:gchhistory}.} The full resolution to the question of {\rm CH} in $\ZFC$ had to wait for Cohen's development of forcing in 1963, which could be used to construct set-theoretic universes with arbitrarily large sizes of the continuum. \Godel's and Cohen's results together finally established the independence of {\rm CH} from $\ZFC$. A decade later, building on advances in forcing techniques, Easton showed that, assuming \GCH, any class function $F$ on the regular cardinals satisfying $F(\alpha)\leq F(\beta)$ for $\alpha\leq\beta$ and $\alpha<\cf(F(\alpha))$ can be realized as the continuum function in a cofinality preserving forcing extension~\cite{easton:gch}, so that in the extension $2^\delta=F(\delta)$ for all regular cardinals $\delta$. Thus, any desired monotonous function on the regular cardinals satisfying the necessary constraints of the Zermelo-\Konig\ inequality could be realized as the continuum function in some set-theoretic universe.\footnote{The situation with singular cardinals turned out to be much more complicated. Silver showed, for example, that if $\delta$ is a singular cardinal of an uncountable cofinality and $2^\alpha=\alpha^+$ for all $\alpha<\delta$, then $2^\delta=\delta^+$ and thus, there is not the same extent of freedom for the continuum function on singular cardinals~\cite{silver:singular}.}

For some simple and other more subtle reasons, the presence of large cardinals in a set-theoretic universe imposes additional constraints on the continuum function, the most obvious of these being that the continuum function must have a closure point at any inaccessible cardinal. Other restrictions arise from large cardinals with strong reflecting properties. For instance, \GCH\ cannot fail for the first time at a measurable cardinal, although Levinski showed in~\cite{levinski:gch} that \GCH\ can hold for the first time at a measurable cardinal. Supercompact cardinals impose much stronger constraints on the continuum function. If $\kappa$ is supercompact and \GCH\ holds below $\kappa$, then it must hold everywhere and, in contrast to Levinski's result, if \GCH\ fails for all regular cardinals below $\kappa$, then it must fail for some regular cardinal $\geq\kappa$~\cite{jech:settheory}.\footnote{Interestingly, in the absence of the axiom of choice, the existence of measurable or supercompact cardinals does not impose any of these restrictions on the continuum function~\cite{apter:choicelesssupercompact}.}  Additionally, certain continuum patterns at a large cardinal can carry increased consistency strength as, for instance, a measurable cardinal $\kappa$ at which  \GCH\ fails has the consistency strength of a measurable cardinal of Mitchell order $o(\kappa)=\kappa^{++}$~\cite{gitik:measurablenotCH}. Some global results are also known concerning sufficient restrictions on the continuum function in universes with large cardinals. Menas showed in~\cite{menas:indes} that, assuming \GCH, there is a cofinality preserving and supercompact cardinal preserving forcing extension realizing any \emph{locally definable}\footnote{A function $F$ is \emph{locally definable} if there is a true sentence $\psi$ and a formula $\varphi(x,y)$ such that for all cardinals $\gamma$, if $H_\gamma\models \psi$, then $F$ has a closure point at $\gamma$ and for all $\alpha,\beta<\gamma$, we have $F(\alpha)=\beta\leftrightarrow H_\gamma\models \varphi(\alpha,\beta)$.} function on the regular cardinals obeying the constraints of Easton's theorem, and Friedman and Honzik extended this result to strong cardinals using generalized Sacks forcing~\cite{syfriedman:continuum}.  In \cite{cody:dissertation}, Cody showed that if $\GCH$ holds, and if $F$ is any function obeying the constraints of Easton's theorem ($F$ need not be locally definable) such that each Woodin cardinal is closed under $F$, then there is a cofinality preserving forcing extension realizing $F$ to which all Woodin cardinals are preserved.

In this article, we show that, assuming \GCH, if $\kappa$ is a Ramsey or a strongly Ramsey cardinal, then any class function on the regular cardinals having a closure point at $\kappa$ and obeying Easton's constraints is realized as the continuum function in a cofinality preserving forcing extension in which $\kappa$ remains Ramsey or strongly Ramsey respectively. In particular, this extends Levinski's result mentioned earlier to Ramsey and strongly Ramsey cardinals. Strongly Ramsey cardinals, introduced by Gitman in~\cite{gitman:ramsey}, fall in between Ramsey cardinals and measurable cardinals in consistency strength, and we shall review their properties in Section~\S\ref{sec:ramsey}.
\begin{main}
Assuming \GCH, if $\kappa$ is a Ramsey or a strongly Ramsey cardinal and $F$ is a class function on the regular cardinals having a closure point at $\kappa$ and satisfying $F(\alpha)\leq F(\beta)$ for $\alpha\leq\beta$ and $\alpha<\emph{cf}(F(\alpha))$, then there is a cofinality preserving forcing extension in which $\kappa$ remains Ramsey or strongly Ramsey respectively, and $F$ is realized as the continuum function, namely $2^\delta=F(\delta)$ for every regular cardinal $\delta$.
\end{main}
\noindent The main theorem will be established by proving Theorems~\ref{th:stronglyramsey} and \ref{th:ramsey} in Section~\S\ref{se:proofofmaintheorem}.

As a corollary to the proof of the main theorem, we have the following.

\begin{corollary}
Suppose $A$ is a class of Ramsey (or strongly Ramsey) cardinals, and $F$ is a function as in the above theorem such that each $\kappa\in A$ is closed under $F$. Then, assuming $\GCH$, there is a cofinality-preserving forcing extension in which each $\kappa\in A$ remains Ramsey (or strongly Ramsey) and $2^\gamma=F(\gamma)$ for each regular cardinal $\gamma$.
\end{corollary}

\section{Ramsey and strongly Ramsey cardinals}\label{sec:ramsey}
Ramsey cardinals $\kappa$ were by introduced by \Erdos\ and Hajnal in~\cite{erdos:ramsey} as having the property that every coloring $f:[\kappa]^{\lt\omega}\to 2$ has a homogeneous set of size $\kappa$. A little explored elementary embeddings characterization of Ramsey cardinals appeared two decades later in~\cite{mitchell:ramsey} and~\cite{dodd:coremodel}. In~\cite{gitman:ramsey}, Gitman generalized aspects of the Ramsey elementary embeddings to introduce several new large cardinal notions, including the strongly Ramsey cardinals. While measurable cardinals and most stronger large cardinal notions are associated with elementary embeddings of the entire universe $V$ into a transitive inner model, Ramsey cardinals and several smaller large cardinals, including weakly compact, unfoldable, and indescribable cardinals, are characterized by the existence of elementary embeddings of transitive set-sized models of set theory without powerset. The theory $\ZFC$ without powerset, denoted $\ZFC^-$, consists of the usual axioms of $\ZFC$ excluding powerset, with the replacement scheme replaced by the collection scheme and the axiom of choice replaced by the well-ordering principle.\footnote{Zarach showed that without powerset, replacement is not equivalent to collection~\cite{Zarach1996:ReplacmentDoesNotImplyCollection} and the axiom of choice does imply the well-ordering principle~\cite{zarach:unions_of_zfminus_models}, while~\cite{zfcminus:gitmanhamkinsjohnstone} showed that collection is required for many of the widely-used set theoretic results such as the \Los\ theorem.} A $\ZFC^-$ model is
called a \emph{weak $\kappa$-model} if it is a transitive set of size $\kappa$
having $\kappa$ as an element; if the model is additionally closed under sequences of length less than $\kappa$ it is called a \emph{$\kappa$-model}. The sets $H_{\theta^+}$, consisting of all sets of hereditary size at most $\theta$, are natural models of $\ZFC^-$, and many weak $\kappa$-models and $\kappa$-models of interest arise as elementary substructures of these. The large cardinal notion with the
simplest characterization involving embeddings between weak
$\kappa$-models is that of weak compactness. Weakly compact cardinals can be defined in a multitude of different ways, among them, the coloring property: $\kappa$ is weakly compact if every coloring $f:[\kappa]^2\to 2$ has a homogeneous set of size $\kappa$, or the tree property: $\kappa$ is inaccessible and every $\kappa$-tree has a cofinal branch~\cite{jech:settheory}. The elementary embeddings property of weakly compact cardinals is that $\kappa$ is weakly compact if $2^{\ltkappa}=\kappa$ and every $A\of\kappa$ is contained in a weak $\kappa$-model $M$ for which there exists an elementary embedding $j:M\to N$ with critical point $\kappa$. In fact, it is equivalent that there must be such an elementary embedding for every $\kappa$-model and, indeed, for every transitive $\ZFC^-$-model of size $\kappa$.\footnote{The strongest of the characterizations follows by using the tree property to obtain a countably complete ultrafilter for the subsets of $\kappa$ in $M$, and the weakest suffices to imply inaccessibility together with the tree property.}

If $j:V\to M$ is any elementary embedding of the universe with critical point some cardinal $\kappa$, then $U=\{A\subseteq\kappa\mid \kappa\in j(A)\}$ is a normal ultrafilter on $\kappa$. The ultrapower of the universe by $U$ is well-founded since $U$ is countably complete and can hence be collapsed to a transitive inner model. Indeed, the ultrapower construction with $U$ may be iterated to produce an $\ORD$-length directed system of  elementary embeddings of transitive inner models. The successor stages of the iteration are ultrapowers by the image of the ultrafilter from the previous stage and direct limits are taken at limit stages. Gaifman showed in~\cite{gaifman:ultrapowers} that if the ultrafilter is countably complete, then all the \emph{iterated ultrapowers} are well-founded. Thus, countably complete ultrafilters are \emph{iterable} in the sense that they produce all well-founded iterated ultrapowers. If $j:M\to N$ is an elementary embedding of a weak $\kappa$-model with critical point $\kappa$, then $U=\{A\subseteq\kappa\mid\kappa\in j(A)\}$ is called an $M$-ultrafilter and is a normal ultrafilter from the perspective of $M$. More precisely, if $\kappa$ is a cardinal in a model $M$\footnote{Unless we specifically state otherwise, all models we work with should be assumed to be transitive.}  of $\ZFC^-$, then $U\of P(\kappa)\cap M$ is an $M$-\emph{ultrafilter}  if the structure $\la M,\in, U\ra$, consisting of $M$ together with a predicate for $U$, satisfies that $U$ is a normal ultrafilter\footnote{Here, we adopt the convention that an ultrafilter on a cardinal $\kappa$ includes the tail sets. It follows that ultrafilters are necessarily non-principal and a normal ultrafilter on $\kappa$ is $\kappa$-complete.} on $\kappa$. Using an $M$-ultrafilter, we can construct the ultrapower of $M$ out of functions on $\kappa$ that are elements of $M$ and this ultrapower satisfies \Los. If the $M$-ultrafilter arose from an elementary embedding $j:M\to N$, as above, then the ultrapower will be well-founded as it embeds into $N$. Alternatively, the ultrapower by an $M$-ultrafilter will be well-founded if it is $\omega_1$-\emph{intersecting}, where every countable collection of sets in the ultrafilter has a non-empty intersection.\footnote{Such ultrafilters are commonly referred to in the literature as \emph{countably complete}, but we find this terminology confusing since in the case of $M$-ultrafilters the intersection is not required to be an element of the ultrafilter but only needs to be non-empty.}  Otherwise, since the $M$-ultrafilter is only countably complete for sequences in $M$ and $M$ might be missing many countable sequences, there is no reason to suppose the ultrapower to be well-founded. Even if the $M$-ultrafilter produces a well-founded ultrapower, iterating the construction requires extra assumptions, since if it is external to $M$, the ultrafilters for the successor stages of the iteration cannot be obtained by taking the image of the ultrafilter in the preceding stage. The construction to define successor stage ultrafilters requires the $M$-ultrafilter to be \emph{weakly amenable} to $M$, that is the intersection of the ultrafilter with any set of size $\kappa$ in $M$ must be an element of $M$. The weak amenability of an $M$-ultrafilter $U$ on $\kappa$ allows us to define a weakly amenable $\Ult(M,U)$-ultrafilter $W=\{[f]_U\mid \{\xi<\kappa\mid f(\xi)\in U\}\in U\}$.\footnote{See Section~19 of~\cite{kanamori:higher} for details.} An elementary embedding $j:M\to N$ with critical point $\kappa$ and the additional property that $M$ and $N$ have the same subsets of $\kappa$ may be used to obtain a weakly amenable $M$-ultrafilter with a well-founded ultrapower. Gitman named such embeddings $\kappa$-\emph{powerset preserving}, and their existence is actually equivalent to the existence of weakly amenable $M$-ultrafilters with well-founded ultrapowers, as the ultrapower map by a weakly amenable $M$-ultrafilter is $\kappa$-powerset preserving. Weak amenability allows for the iterated ultrapower construction to proceed but it does not guarantee well-foundedness of the iterated ultrapowers. Weakly amenable $M$-ultrafilters with well-founded ultrapowers span the full spectrum of iterability behavior from having only the well-founded ultrapower to having exactly $\alpha$-many well-founded iterated ultrapowers for any $\alpha<\omega_1$ to being fully iterable~\cite{gitman:welch}.\footnote{An $M$-ultrafilter that produces $\omega_1$-many iterated ultrapowers is already fully iterable by a theorem of Gaifman from~\cite{gaifman:ultrapowers}.} Kunen showed in~\cite{kunen:ultrapowers} that being $\omega_1$-intersecting is a sufficient condition for a weakly amenable $M$-ultrafilter to be fully iterable, which leads us into the elementary embeddings characterization of Ramsey cardinals.
\begin{theorem}[\cite{mitchell:ramsey, dodd:coremodel}]\label{th:ramseyult}
A cardinal $\kappa$ is Ramsey if and only if every $A\of\kappa$ is contained in a weak $\kappa$-model $M$ for which there exists a weakly amenable $\omega_1$-intersecting $M$-ultrafilter on $\kappa$.
\end{theorem}
Recall that while weakly compact cardinals are characterized by the existence of elementary embeddings on weak $\kappa$-models, it is equivalent to strengthen the characterization to the existence of embeddings on $\kappa$-models and indeed on all transitive $\ZFC^-$-models of size $\kappa$. This is not the case with Ramsey cardinals. Once the embeddings are required to be $\kappa$-powerset preserving, the reflection properties this introduces causes the strengthening to $\kappa$-models to produce a jump in consistency strength. In~\cite{gitman:ramsey}, Gitman introduced the notion of strongly Ramsey cardinals, using precisely this strengthened characterization, and showed that they are limits of Ramsey cardinals, but still weaker than measurable cardinals which are limits of them. She also showed there that requiring the existence of a $\kappa$-powerset preserving embedding for every transitive $\ZFC^-$-model of size $\kappa$ is inconsistent.
\begin{definition}
A cardinal $\kappa$ is \emph{strongly Ramsey} if every $A\of\kappa$ is contained in a $\kappa$-model $M$ for which there exists a weakly amenable $M$-ultrafilter on $\kappa$.
\end{definition}
\noindent Note that if $M$ is a $\kappa$-model, then any $M$-ultrafilter is automatically $\omega_1$-intersecting since $M$ is closed under countable sequences.
The elementary embeddings characterization of strongly Ramsey cardinals can be reformulated without $M$-ultrafilters, using the equivalence of the existence of weakly amenable ultrafilters on $\kappa$ and $\kappa$-powerset preserving embeddings mentioned earlier.
\begin{remark}
A cardinal $\kappa$ is strongly Ramsey if and only if every $A\of\kappa$ is contained in a $\kappa$-model for which there exists a $\kappa$-powerset preserving elementary embedding.
\end{remark}

In the case of both Ramsey and strongly Ramsey cardinals, we will show how to strengthen their respective characterizations by requiring the existence of  elementary embeddings for models of full \ZFC, additionaly having the properties described in Definition~\ref{def:special} below.
\begin{definition}\label{def:special}
A weak $\kappa$-model $M$ is $\alpha$-\emph{special} if it is the union of a continuous elementary chain of (not necessarily transitive) submodels $\la X_\xi\mid \xi<\alpha\ra$ such that each $X_\xi$ is a set of size $\kappa$ in $M$  and each $X_{\xi+1}^{\ltkappa}\subseteq X_{\xi+1}$ in $M$.
\end{definition}

Before we proceed, let us recall two useful standard properties of ultrapower maps.
\begin{remark}\label{rem:ultrapowers}
Suppose $M$ is a weak $\kappa$-model.
\begin{itemize}
\item[(1)] An elementary embedding $j:M\to N$ is the ultrapower map
by an $M$-ultrafilter on $\kappa$ if and
only if $N=\{j(f)(\kappa)\mid f:\kappa\to M,f\in M\}$.
\item[(2)] If $j:M\to N$ is the ultrapower map by an $M$-ultrafilter on
$\kappa$ and $M^{\alpha}\of M$ for some $\alpha<\kappa$, then
$N^{\alpha}\of N$. In particular, if $M$ is a $\kappa$-model, then $N$ is a $\kappa$-model as well.
\end{itemize}
\end{remark}
\begin{lemma}\label{le:specialult}
If $M$ is an $\alpha$-special weak $\kappa$-model and $N$ is the ultrapower of $M$ by a weakly amenable $M$-ultrafilter on $\kappa$, then $N$ is $\alpha$-special as well.
\end{lemma}
\begin{proof}
Fix an $\alpha$-special weak $\kappa$-model $M$ with the ultrapower map $j:M\to N$ by a weakly amenable $M$-ultrafilter $U$ on $\kappa$, and let $\la X_\xi\mid \xi<\alpha\ra$ be the elementary chain of submodels witnessing that $M$ is $\alpha$-special. Since $M=\Union_{\xi<\alpha}X_\xi$, there is some $\eta$ such that $\kappa\in X_\eta$, and so by taking the tail sequence starting at $\eta$, we may assume without loss of generality that $\kappa$ is in every $X_\xi$.  For $\xi<\alpha$, we define $Y_\xi=\{j(f)(\kappa)\mid f:\kappa\to X_\xi,f\in X_\xi\}$ and argue that the sequence $\la Y_\xi\mid\xi<\alpha\ra$ witnesses that $N$ is $\alpha$-special. Observe that each $j\image X_\xi$ is an element of $N$ of size $\kappa$ since it is the range of $j(f)\restrict\kappa$ for any bijection $f\in M$ between $\kappa$ and $X_\xi$. It follows that each $Y_\xi$ is an element of $N$ of size $\kappa$.  It is clear that the $Y_\xi$ form a continuous increasing chain with union $N$. Since the chain is increasing, to show that it is elementary it suffices to see that the $Y_\xi$ are elementary in $N$. For this, we observe that each $Y_\xi$ is isomorphic to the ultrapower of $X_\xi$ by $U$ that uses functions $f:\kappa\to X_\xi$ from $X_\xi$ (this needs that $\kappa\in X_\xi$), and so we have $Y_\xi\models\varphi(j(f)(\kappa))\leftrightarrow\{\xi<\kappa\mid X_\xi\models \varphi(f(\xi))\}=\{\xi<\kappa\mid M\models \varphi(f(\xi))\}\in U\leftrightarrow N\models\varphi(j(f)(\kappa))$. It remains to argue that the $Y_{\xi+1}$ are closed under $\ltkappa$-sequences in $N$. Towards this end, we fix a $Y_{\gamma+1}$ and let $\vec y=\la y_\eta\mid \eta<\delta\ra$ for some $\delta<\kappa$ be a sequence of elements of $Y_{\gamma+1}$ in $N$. Since each $y_\eta\in Y_{\gamma+1}$, we may fix a function $g_\eta\in X_{\gamma+1}$ such that $j(g_\eta)(\kappa)=y_\eta$. It is tempting at this point to conclude that the sequence $\la g_\eta\mid\eta<\delta\ra$ belongs to $X_{\gamma+1}$ since it is $\ltkappa$-closed in $M$, but for this to be the case, we must first argue that $\la g_\eta\mid\eta<\delta\ra$ is an element of $M$, which there is no obvious reason to believe. Instead, we fix $F:\kappa\to M$ such that $j(F)(\kappa)=\vec y$ and assume by choosing a $U$-equivalent function, if necessary, that for every $\beta<\kappa$, $F(\beta)$ is a $\delta$-sequence. Thus, there is a sequence $\la f_\eta\mid\eta<\delta\ra\in M$ of functions such that $f_\eta(\beta)=F(\beta)_\eta$ for all $\beta<\kappa$ and $\eta<\delta$. By \Los, it follows  that $j(f_\eta)(\kappa)=j(g_\eta)(\kappa)=y_\eta$, but the functions $f_\eta$ need not be elements of $X_{\gamma+1}$. Nevertheless, using $\la f_\eta\mid\eta<\delta\ra$, the model $M$ will be able to construct a sequence of $U$-equivalent functions $h_\eta$ that will be elements of $X_{\gamma+1}$. By the weak amenability of $U$, the set $u=\{A_\eta\mid\eta<\delta\}$, where $A_\eta=\{\xi<\kappa\mid f_\eta(\xi)=g_\eta(\xi)\}$, is an element of $M$. The functions $g_\eta$ witness that $M$ satisfies that for each $\eta<\delta$, there is $h_\eta\in X_{\gamma+1}$ that is $u$-equivalent to $f_\eta$. Now, using collection and choice in $M$, we obtain a sequence $\la h_\eta\mid\eta<\delta\ra\in M$ of functions $h_\eta\in X_{\gamma+1}$ such that $j(h_\eta)(\kappa)=j(f_\eta)(\kappa)=y_\eta$. By the $\ltkappa$-closure of $X_{\gamma+1}$ in $M$, the sequence $\la h_\eta\mid\eta<\delta\ra$ is an element of $X_{\gamma+1}$, from which it easily follows that $\vec y\in Y_{\gamma+1}$.

\end{proof}

\begin{lemma}\label{le:strongramseycharacterization}
If $\kappa$ is strongly Ramsey, then every $A\of\kappa$ is contained in a $\kappa$-special $\kappa$-model $M$ for which there exists a weakly amenable $M$-ultrafilter on $\kappa$ with the ultrapower map $j:M\to N$ having $M=V_{j(\kappa)}^N$. Consequently, $M\models\ZFC$ and $M\in N$.
\end{lemma}
\begin{proof}
Fix $A\of\kappa$ and a $\kappa$-model $\Mbar$ containing $A$ for which there exists a weakly amenable $\Mbar$-ultrafilter $U$ on $\kappa$ with the ultrapower map $\bar j:\Mbar\to \Nbar$. Note that since $\Mbar$ is a $\kappa$-model, we automatically have $V_\kappa\in \Mbar$. Also, $\Nbar$ is a $\kappa$-model by Remarks~\ref{rem:ultrapowers}~(2). It is easy to see that $U$ remains a weakly amenable ultrafilter for the $\kappa$-model $\her{\kappa}^{\Mbar}=\{B\in \Mbar\mid |\text{Trcl}(B)|\leq\kappa\}$, which lets us assume, without loss of generality, that $\Mbar=\her{\kappa}^{\Mbar}$. It now follows by weak amenability that $\Mbar=\her{\kappa}^{\Nbar}$ and hence $\Mbar\in \Nbar$. Let $W$ be the $\Nbar$-ultrafilter for the second stage of the iteration and note that $\Ult(\Nbar, W)$ is well-founded as $U$ is $\omega_1$-intersecting and hence iterable. It is a standard fact from the theory of iterated ultrapowers that $\Ult(\Nbar, W)\cong\Ult(\Nbar, U)$, which must then be well-founded. So we let $j:\Nbar\to K$ be the ultrapower map of $\Nbar$ by $U$. Using that $\Mbar\in \Nbar$, we may restrict $j$ to $j:\Mbar \to j(\Mbar)$, and observe that, since $\Mbar$ contains all functions $f:\kappa\to\Mbar$ of $\Nbar$, it is precisely the ultrapower of $\Mbar$ by $U$. Thus, $j(\Mbar)=\Nbar=\her{j(\kappa)}^K$ and $\bar j=j$ on $\Mbar$, from which it follows that $\bar j(V_\kappa)=V_{\bar j(\kappa)}^\Nbar=V_{ j(\kappa)}^K\of \Nbar$. Now we let $M=V_{j(\kappa)}^K$ and $N=j(M)$. Since $M=V_{\bar j(\kappa)}^{\Nbar}$, subsets of size $\kappa$ of $M$ in $\Nbar$ are elements of $M$, and so, in particular, $\Nbar$ thinks that $M$ is a $\kappa$-model and it is correct about this by virtue of being a $\kappa$-model itself. Consider the restriction of $j$ to $j:M\to N$, which is precisely the ultrapower of $M$ by $U$ as  $M$ contains all functions $f:\kappa\to M$ of $K$. It is clear that $M=V_{j(\kappa)}^N$ and $A\in M$. It  remains to find a sequence of submodels $\la X_\xi\mid\xi<\kappa\ra$ witnessing that $M$ is $\kappa$-special.

Externally, we fix an enumeration $M=\{a_\xi\mid\xi<\kappa\}$. Inside $N$, we build $X_0\prec M$ of size $\kappa$ such that $a_0\in X_0$ and $X_0^{<\kappa}\subseteq X_0$. Note that $X_0\in M$ and $M$ knows that it has size $\kappa$ since subsets of $V_{j(\kappa)}^N$ of size $\kappa$ are elements of
$V_{j(\kappa)}^N$ by the inaccessibility of $j(\kappa)$ in $N$. It is also clear that $X_0^{\ltkappa}\subseteq X_0$ in $M$. Now suppose inductively that we have $X_\xi\prec M$ and $X_\xi$ is an element of size $\kappa$ in $M$. Inside $N$, we build $X_{\xi+1}\prec M$ of size $\kappa$ such that $X_\xi,a_\xi\in X_{\xi+1}$ and $X_{\xi+1}^{\ltkappa}\subseteq X_{\xi+1}$. Again, it follows that $X_{\xi+1}$ is an element of size $\kappa$ in $M$ and $X_{\xi+1}^{\ltkappa}\subseteq X_{\xi+1}$ in $M$. At limit stages $\lambda$, we let $X_\lambda=\Union_{\xi<\lambda} X_\xi$, and note that $X_\lambda$ is an element of size $\kappa$ in $M$ since it is closed under $\ltkappa$-sequences. It is clear that the sequence $\la X_\xi\mid \xi<\kappa\ra$ witnesses that $M$ is $\kappa$-special.
\end{proof}
\begin{lemma}\label{le:ramseycharacterization}
If $\kappa$ is Ramsey, then every $A\of\kappa$ is contained in an $\omega$-special weak $\kappa$-model $M$ for which there exists a weakly amenable $\omega_1$-intersecting $M$-ultrafilter on $\kappa$ with the ultrapower map $j:M\to N$ having $M=V_{j(\kappa)}^N$. Consequently, $M\models\ZFC$ and $M\in N$.
\end{lemma}
\begin{proof}
Fix a weak $\kappa$-model $\overline{\Mbar}$ containing $A$ and $V_\kappa$ for which there exists a weakly amenable $\omega_1$-intersecting
$\overline{\Mbar}$-ultrafilter $\overline U$ on $\kappa$ with the ultrapower map $\bar{\bar{j}}:\overline{\Mbar}\to\overline{\Nbar}$. As in the proof of Lemma~\ref{le:strongramseycharacterization}, we may assume that $\overline \Mbar=\her{\kappa}^{\overline\Nbar}$. Inside $\overline\Nbar$, we build a transitive $M_0\prec \overline\Mbar$ of size $\kappa$ with $M_0^{\ltkappa}\of M_0$. Since $\overline \Mbar=\her{\kappa}^{\overline\Nbar}$, it is clear that $M_0$ is an element of size $\kappa$ in $\overline{\Mbar}$. Now suppose inductively that we have $M_n\prec \overline\Mbar$ and $M_n$ is an element of size $\kappa$ in $\overline\Mbar$. Inside $\overline \Nbar$, we build a transitive $M_{n+1}\prec \overline \Mbar$ of size $\kappa$ such that $M_n\in M_{n+1}$, $U\cap M_n\in M_{n+1}$ (using weak amenability), and $M_{n+1}^{\ltkappa}\subseteq M_{n+1}$. Now we let $\Mbar=\Union_{n<\omega}M_n$. The model $\Mbar$ is $\omega$-special by construction and $\Mbar=\her{\kappa}^{\Mbar}$. The fact that $U\cap M_n$ are elements of $\Mbar$ ensures that $U=\overline U\cap\Mbar$ is a weakly amenable $\Mbar$-ultrafilter, and it is also clearly $\omega_1$-intersecting. So we let $\bar j:\Mbar\to \Nbar$ be the ultrapower map by $U$. By Lemma \ref{le:specialult}, the ultrapower $\Nbar$ is $\omega$-special, and so we may fix a sequence $\la \overline X_n\mid n<\omega\ra$ witnessing this. Now we exactly follow the proof of Lemma \ref{le:strongramseycharacterization} to we obtain the ultrapower map $j:M\to N$ by $U$ with $M=V_{j(\kappa)}^N=V_{\bar j(\kappa)}^\Nbar$. For $n<\omega$, we let $X_n=\overline X_n\cap V_{\bar j(\kappa)}^\Nbar$ and argue that a tail segment of the sequence $\la X_n\mid n<\omega\ra$ witnesses that $M$ is $\omega$-special. Each $X_n$ is an element of $M=V_{\bar j(\kappa)}^\Nbar$ since it is a subset of it of size $\kappa$, and for the same reason, the $X_n$ are closed under $\ltkappa$-sequences in $M$. Once $\bar j(\kappa)\in \overline X_n$, it is easy to see that for all formulas $\varphi(x)$ and $a\in X_n$, we have that $X_n\models \varphi(a)$ if and only if $\overline X_n\models ``V_{\bar j(\kappa)}^{\Nbar}\models \varphi(a)"$. Thus, $X_n\prec M=V_{\bar j(\kappa)}^\Nbar$ for all $n$ such that $\bar j(\kappa)\in \overline X_n$.
\end{proof}
\section{Forcing preliminaries}\label{se:forcingpreliminaries}
The proof of the main theorem will proceed, as do most standard indestructibility arguments for large cardinals, by lifting elementary embeddings characterizing the cardinal from the ground model to the forcing extension.\footnote{An embedding $j:M\to N$ is said to lift to another embedding $j^*:M^*\to N^*$, where
$M\of M^*$ and $N\of N^*$, if the two embeddings agree on the smaller domain, that is $j^*\restrict M = j$.} Below we review the lifting techniques which will be used in the proof.

The \emph{lifting criterion} states that in order to lift an embedding $j:M\to N$ to the forcing extension $M[G]$ by an $M$-generic filter $G\subseteq \p$ for some $\p\in M$, we must find an $N$-generic filter $H$ for the image poset $j(\p)$ containing $j\image G$ as a subset.
\begin{lemma}[Lifting Criterion]\label{le:liftingcriterion} Suppose
that $M$ is a model of $\ZFC^-$ and \hbox{$j:M\to N$} is an elementary embedding. If $M[G]$ and $N[H]$ are generic extensions by forcing notions $\p$
and $j(\p)$ respectively, then the embedding $j$ lifts to an embedding
$j:M[G]\to N[H]$ with $j(G)=H$ if and only if $j\image G\of H$.
\end{lemma}
\noindent If $j$ happens to be an ultrapower map by some $M$-ultrafilter $U$, then it is not difficult to see, using Remarks~\ref{rem:ultrapowers}~(1), that the lift $j:M[G]\to N[H]$ is again an ultrapower map by an $M[G]$-ultrafilter extending $U$.

Typically, we construct the $N$-generic filter $H\supseteq j\image G$ for $j(\p)$ using the \emph{diagonalization criterion}. The diagonalization criterion is a generalization to larger models of the standard diagonalization argument for constructing generic filters for countable models of set theory.
\begin{lemma}[Diagonalization Criterion]\label{le:diag(1)}
If $\p$ is a forcing notion in a model $M$ of $\ZFC^-$ and for some
cardinal $\kappa$ the following criteria are satisfied:
\begin{itemize}
\item[(1)] $M^{\ltkappa}\of M$,
\item[(2)] $\p$ is $\ltkappa$-closed in $M$,
\item[(3)] $M$ has at most $\kappa$ many antichains of $\p$,
\end{itemize}
then there is an $M$-generic filter for $\p$.
\end{lemma}

The diagonalization criterion guarantees the existence of an $N$-generic filter for any $\ltkappa$-closed poset in a $\kappa$-model $N$, making it an indispensable component of indestructiblity arguments involving strongly Ramsey cardinals whose elementary embeddings have $\kappa$-model targets by Remarks~\ref{rem:ultrapowers}~(2). But, because of the closure requirement on the target model, the diagonalization criterion is not applicable to the Ramsey elementary embeddings.
For this reason, \cite{gitman:ramseyindes} introduced diagonalization criterion (2) that, at the expense of extra assumptions, does not require any closure on the model.
\begin{lemma}[Diagonalization Criterion (2)]\label{le:diag(2)}
If $\p$ is a forcing notion in a model $M$ of $\ZFC^-$ and for some
cardinal $\kappa$  the following criteria are satisfied:
\begin{itemize}
\item[(1)] $\p$ is $\leqkappa$-closed in $M$,
\item[(2)] $M=\Union_{n<\omega}X_n$ for some sequence $\la X_n\mid n<\omega\ra$ such that each $X_n$ is
 an element of size $\kappa$ in $M$,
\end{itemize}
then there is an $M$-generic filter for $\p$.
\end{lemma}
\begin{proof}
Since $\p$ is $\leqkappa$-closed in $M$ and $X_0$ has size $\kappa$ in $M$, working inside $M$, we construct a $\kappa$-length descending sequence of conditions meeting all dense subsets of $\p$ that are elements of $X_0$, and choose a condition $p_0\in \p$ below the sequence. Suppose inductively that we are given a condition $p_n$ having the property that if $D$ is a dense subset of $\p$ in $X_n$, then there is a condition above $p_n$ meeting $D$. Working inside $M$, we construct, below $p_n$, a descending $\kappa$-length sequence of conditions meeting all dense sets of $\p$ that are elements of $X_{n+1}$, and choose a condition $p_{n+1}\in \p$ below the sequence. Since $M=\Union_{n<\omega}X_n$, any filter $G$ generated by the sequence $\la p_n\mid n<\omega\ra$ is $M$-generic.
\end{proof}
\noindent Diagonalization criterion~(2) guarantees the existence of an $N$-generic filter for any $\lesseq\kappa$-closed poset in an $\omega$-special weak $\kappa$-model $N$. Since Ramsey elementary embeddings can be assumed to have $\omega$-special targets by Lemma~\ref{le:ramseycharacterization}, diagonolization criterion~(2) will play the role of the diagonalization criterion in lifting these embeddings.

It is also often necessary in the course of lifting arguments to be able to verify that a forcing extension $M[G]$ remains as closed in its overarching universe as the ground model $M$ was in its universe. For instance, to verify that the lift $j:M[G]\to N[H]$ of a ground model embedding $j:M\to N$ is one of the embeddings witnessing that $\kappa$ is strongly Ramsey in $V[G]$, we would have to argue, in particular, that $M[G]$ is closed under $\ltkappa$-sequences in $V[G]$. Below, we recall two standard facts about preserving the closure of a model of set theory to its forcing extension. In order to make the facts applicable to nontransitive models as well, we will first review the relevant notion of $X$-genericity for a nontransitive model $X\prec M\models \ZFC^-$. We will need the closure preservation for nontransitive models to argue that a forcing extension of an $\alpha$-special model is again $\alpha$-special by using as the witnessing submodels the forcing extensions of the ground model ones.

\begin{definition}\label{def:xgeneric}
Suppose $X\prec M\models \ZFC^-$ and $\p\in X$ is a
forcing notion, then a filter $G\of \p$ is $X$-\emph{generic} if $G\cap D\cap
X\neq \emptyset$ for every dense subset $D\in X$ of $\p$.
\end{definition}
\noindent Note that the usual definition of genericity and Definition \ref{def:xgeneric} coincide if $X$ happens to be transitive.
If $X\prec M\models\ZFC^-$ and $G$ is an $M$-generic filter for a poset $\p\in X$, then we define $X[G]=\{\tau_G\mid \tau\in X\text{ is a }\p\text{-name}\}$.
\begin{remark}\label{rem:xgenericelem}
If $X\prec M\models \ZFC^-$ and $G$ is an $M$-generic filter for a poset $\p\in X$ which is additionally $X$-generic, then $X[G]\prec M[G]$.
\end{remark}

\begin{lemma}[Ground Closure
Criterion]\label{le:groundclosure} Suppose $X\models \ZFC^-$, for some ordinal $\gamma$,
$X^{\gamma}\of X$ and there is in $M$, an $X$-generic filter $H\of
\p$ for a forcing notion $\p\in X$. Then $X[H]^{\gamma}\of X[H]$.
\end{lemma}
\begin{lemma}[Generic Closure Criterion]\label{le:genericclosure}
Suppose $X\prec M\models\ZFC^-$ and
for some ordinal $\gamma$, $X^{\lt\gamma}\of X$ in $M$. If $G\of \p$ is
$M$-generic for a forcing notion $\p\in X$ such that $\p\of X$ and
has $\gamma$-cc in $M$, then $X[G]^{\lt\gamma}\of X[G]$ in $M[G]$.
\end{lemma}

The forcing we will use to prove the main main results of this paper will be an Easton support iteration of Easton support products (see the definition of $\P^F$ at the beginning of Section~\S\ref{se:proofofmaintheorem}). To lift embeddings through this forcing we will need the following lemma.

\begin{lemma}[Easton's Lemma]\label{le:easton}
Suppose $M$ is a model of $\ZFC$ \emph{(}or $\ZFC^-$\emph{)} and $\P,\Q \in M$ are forcing notions. If $\P$ has the $\kappa^+$-cc and $\Q$ is ${\leq}\kappa$-closed in $M$, then $\Q$ remains $\lesseq\kappa$-distributive in $M^{\p}$.
\end{lemma}


We end this section by recalling a standard fact about forcing iterations of an inaccessible length $\kappa$.

\begin{lemma}\label{le:eastoniterations}
Suppose $\p$ is an iteration of inaccessible length $\kappa$ such
that for all $\alpha<\kappa$, $\forces_\alpha \dot{\q}_\alpha\in
\check{V}_\kappa$ and a direct limit is taken on a stationary set of
stages below $\kappa$, then $\p$ has size $\kappa$ and the $\kappa$-cc.
\end{lemma}
A more detailed exposition of standard lifting techniques can be found in
\cite{cummings:weaklycompact}.

\section{Proof of Main Theorem}\label{se:proofofmaintheorem}
The main theorem will be established by proving Theorems~\ref{th:stronglyramsey} and \ref{th:ramsey} below.

In his celebrated result, Easton showed that assuming \GCH, any class function $F$ on the regular cardinals satisfying that $F(\alpha)\leq F(\beta)$ for $\alpha\leq\beta$ and $\alpha<\cf(F(\alpha))$ is realized as the continuum function in a cofinality preserving forcing extension. To construct the forcing extension, Easton used the $\ORD$-length Easton support\footnote{A product forcing is said to have \emph{Easton support} if for every regular cardinal $\delta$  and $p\in\p$, we have $|\text{dom}(p)\cap \delta|<\delta$. A forcing iteration is said to have \emph{Easton support} if direct limits are taken at inaccessible stages and inverse limits are taken elsewhere.} product $\e^F=\Pi_{\delta\in
\ORD}\Add(\delta,F(\delta))$, where $\Add(\delta,\gamma)$ denotes the poset adding $\gamma$-many Cohen subsets to $\delta$ with conditions of size less than $\delta$. Easton's result may be achieved as well by forcing with an Easton support $\ORD$-length iteration of Easton support products. Versions of this iteration, which we describe below, appeared in the arguments of Menas \cite{menas:indes}, Friedman and Honzik \cite{syfriedman:continuum}, and Cody \cite{cody:dissertation} cited in the introduction. The advantage to using an iteration of products is that the standard lifting techniques for preserving large cardinals work well with such iterations.

Let $\P^F=\langle (\P_{\eta},\dot{\Q}_{\eta}) : \eta\in\ORD \rangle$ be the Easton support $\ORD$-length iteration of Easton support products, defined as follows. For a given ordinal $\alpha$, let $\bar{\alpha}$ denote the least inaccessible closure point of $F$ greater than $\alpha$.
\begin{enumerate}
\item[$(1)$] If $\eta$ is an inaccessible closure point of $F$ in $V^{\P_{\eta}}$, then $\dot{\Q}_{\eta}$ is a $\P_{\eta}$-name for the Easton support product
$$\E^F_{[\eta,\bar{\eta})}:=\prod_{\gamma\in[\eta,\bar{\eta})\cap\REG} \Add(\gamma,F(\gamma))$$
as defined in $V^{\P_{\eta}}$ and $\P_{\eta+1}=\P_{\eta}*\dot{\Q}_{\eta}$.
\item[$(2)$] If $\eta$ is a singular closure point of $F$ in $V^{\P_{\eta}}$, then $\dot{\Q}_{\eta}$ is a $\P_{\eta}$-name for
$$\E^F_{(\eta,\bar{\eta})}:=\prod_{\gamma\in(\eta,\bar{\eta})\cap\REG} \Add(\gamma,F(\gamma))$$
as defined in $V^{\P_{\eta}}$ and $\P_{\eta+1}=\P_{\eta}*\dot{\Q}_{\eta}$.
\item[$(3)$] Otherwise, if $\eta$ is not a closure point of $F$, then $\dot{\Q}_\eta$ is a $\P_\eta$-name for trivial forcing and $\P_{\eta+1}=\P_\eta*\dot{\Q}_\eta$.
\end{enumerate}

The forcing for adding $\lambda$-many Cohen subsets to a regular cardinal, which we are denoting by $\Add(\kappa,\lambda)$, can be characterized in many ways. In what follows, we adopt the convention that $\Add(\kappa,\lambda)$ consists of all partial functions $p$ from $\lambda$ to $\Add(\kappa,1)$ such that $|p|<\kappa$.

The next lemma is established by the standard arguments of factoring $\p^F$ at the appropriate stages and counting nice names.

\begin{lemma}\label{le:continuumiteration}
If the \GCH\ holds, then the forcing extension $V[G]$ by $G\of\p^F$ is  cofinality preserving and has $2^\delta=F(\delta)$ for all regular cardinals $\delta$.
\end{lemma}

We will start by proving the main theorem for weakly compact cardinals as a warm-up. In this case the proof is relatively straightforward and it will allow us to anticipate the obstacles that must be overcome to carry out the argument for Ramsey and strongly Ramsey cardinals. The weaker result that the $\GCH$ may fail for the first time at a weakly compact cardinal appears in \cite{cummings:weaklycompact} and the full result is part of the set theoretic folklore. Recall, from Section~\S\ref{sec:ramsey}, that if $\kappa$ is weakly compact, then every $A\of\kappa$ is contained in a $\kappa$-model $M$ for which there is an elementary embedding $j:M\to N$ with critical point $\kappa$, and to demonstrate that $\kappa$ is weakly compact, it suffices that $2^{\ltkappa}=\kappa$ and every $A\subseteq\kappa$ is contained in such a weak $\kappa$-model. First, it is useful to make the following easy observation.
\begin{remark}
If $\kappa$ is a weakly compact, Ramsey, or strongly Ramsey cardinal, then this is preserved to forcing extensions by $\lesseq\kappa$-distributive forcing.
\end{remark}
\begin{theorem}\label{th:weaklycompact}
Assuming \GCH, if $\kappa$ is weakly compact and $F$ is a class function on the regular cardinals having a closure point at $\kappa$ and satisfying $F(\alpha)\leq F(\beta)$ for $\alpha\leq\beta$ and $\alpha<\cf(F(\alpha))$, then there is a confinality preserving forcing extension in which $\kappa$ remains weakly compact and $F$ is realized as the continuum function, that is $2^\delta=F(\delta)$ for every regular cardinal $\delta$.
\end{theorem}
\begin{proof}
We will argue that $\kappa$ remains weakly compact after forcing with the iteration $\p^F$. Observe that we may factor $\P^F$ as:
\begin{align*}
\p^F&\cong \p_\kappa^F*\dot\E^F_{[\kappa,\bar{\kappa})}*\dot\p_\tail^F \\
	&\cong \p_\kappa^F*\left(\dot\Add(\kappa,F(\kappa))\times\dot\E^F_{(\kappa,\bar{\kappa})}\right)*\dot\p^F_\tail.
\end{align*}
It follows from Lemma \ref{le:easton} that $\dot\E^F_{(\kappa,\bar{\kappa})}*\dot\P^F_\tail$ is forced to be $\leqkappa$-distributive by $\P^F_\kappa*\Add(\kappa,F(\kappa))$. Since $\leqkappa$-distributive forcing preserves the weak compactness of $\kappa$, it suffices to prove that $\kappa$ remains weakly compact after forcing with $\p_\kappa^F*\dot\Add(\kappa,F(\kappa))$. Towards this end, we let $G*K\subseteq\p_\kappa^F*\dot\Add(\kappa,F(\kappa))$ be $V$-generic and consider the extension $V[G][K]$. Standard arguments show that $2^{\ltkappa}=\kappa$ in $V[G][K]$. We will show that every $A\of\kappa$ in $V[G][K]$ is contained there in a weak $\kappa$-model for which there exists an elementary embedding with critical point $\kappa$.



We fix $A\of\kappa$ in $V[G][K]$ and choose a nice
$\p_\kappa^F*\dot\Add(\kappa,F(\kappa))$-name $\dot A$ in $V$ such that $(\dot A)_{G*K}=A$. We would like to be able to put $\dot A$ into a $\kappa$-model $M$ in $V$ for which there is an elementary embedding $j:M\to N$ with critical point $\kappa$ and lift $j$ to $M[G][K]$ in $V[G][K]$. However, the name $\dot A$ and the poset $\p_\kappa^F*\dot\Add(\kappa,F(\kappa))$ are both too large to fit into a model of size $\kappa$. Thus, our strategy requires a slight refinement. Since $\Add(\kappa,F(\kappa))$ has the $\kappa^+$-cc, it follows that at most $\kappa$-many ordinals below $F(\kappa)$ can appear in the domains of conditions in $\dot A$ and hence, using an automorphism of $\Add(\kappa,F(\kappa))$, we may assume that all the conditions in $\dot A$ appear on the first coordinate of $\Add(\kappa,F(\kappa))$. Thus, we may assume that $\dot A$ is a nice
$\p_\kappa^F*\dot\Add(\kappa,1)$-name and $(\dot A)_{G*g}=A$ where $g$ is the $V[G]$-generic for $\Add(\kappa,1)$ obtained from the first coordinate of $K$. Now we  fix a $\kappa$-model $M$ containing $\dot A$ and $f=F\restrict\kappa$ for which there is an elementary embedding $j:M\to N$ with
critical point $\kappa$. We may assume without loss of generality that $j$ is the ultrapower by an $M$-ultrafilter on $\kappa$ and therefore, by Remark \ref{th:ramseyult}(2), $N$ is a $\kappa$-model as well. Since $M$ is a $\kappa$-model, $V_\kappa\in M$ and so $\p_\kappa^F*\Add(\kappa,1)=\p_\kappa^f*\Add(\kappa,1)\in M$ as it is definable over $V_\kappa$ using $f$. We shall lift $j$ to $M[G][g]$ in $V[G][K]$, thus witnessing that $A$ can be put into a weak $\kappa$-model with an elementary embedding with critical point $\kappa$.

First, we lift $j$ to $M[G]$ in $V[G][K]$. Using the lifting criterion (Lemma \ref{le:liftingcriterion}), it suffices to find an $N$-generic filter for the poset $j(\p_\kappa^f)$ containing $j\image G$ as a subset. By elementarity, the poset $j(\p_\kappa^f)$ is the $N$-version of the iteration $\p^{j(f)}_{j(\kappa)}$. As $j(f)\restrict\kappa=f$ and $N$ is correct about $\p^f_\kappa$ by virtue of having $V_\kappa$ as an element, we may factor $j(\P^f_\kappa)$ in $N$ as follows. We have
\begin{align}
j(\P^f_\kappa)	&\cong \p_\kappa^f*\left(\dot{\Add}(\kappa,j(f)(\kappa))\times\dot{\tilde\E}^{j(f)}_{(\kappa,\bar{\kappa}^N)}\right)*\dot{\tilde\p}_\tail^{j(f)} \label{eqn:factorization}
\end{align}
where (1) $\bar{\kappa}^N$ denotes the least closure point of $j(f)$ greater than $\kappa$, (2) $\dot{\tilde\E}^{j(f)}_{(\kappa,\bar{\kappa}^N)}$ is a $\P^f_\kappa$-name for the Easton support product ${\tilde\E}^{j(f)}_{(\kappa,\bar{\kappa}^N)}$ as defined in $N^{\P^f_\kappa}$, and (3) $\dot{\tilde\P}^{j(f)}_\tail$ is a name for the tail of the iteration $j(\P^f_\kappa)$. Note that $\p^f_\kappa$ has the $\kappa$-cc by Lemma~\ref{le:eastoniterations}.

We use $G$ as the $N$-generic filter for the $\p^f_\kappa$ part of the forcing to satisfy the lifting criterion. We will build $N[G]$-generic filters, $\tilde{H}_\kappa$ and $\tilde{H}_{(\kappa,\bar{\kappa}^N)}$, for $\Add(\kappa,j(f)(\kappa))^{N[G]}$ and ${\tilde \E}^{j(f)}_{(\kappa,\bar{\kappa}^N)}=(\dot{\tilde \E}^{j(f)}_{(\kappa,\bar{\kappa}^N)})_G$ respectively in $V[G][K]$. Since\break $\Add(\kappa,j(f)(\kappa))^{N[G]}$ has the $\kappa^+$-cc and ${\tilde \E}^{j(f)}_{(\kappa,\bar{\kappa}^N)}$ is ${\leq}\kappa$-closed in $N[G]$ it will then follow that $\tilde{H}_\kappa$ and $\tilde{H}_{(\kappa,\bar{\kappa}^N)}$ are mutually generic. The most subtle part of the argument involves making sure that $\tilde{H}_\kappa$ includes $g$ as a factor, so that $g$ may be used as a master condition in the second stage of the lift. Note that the forcing extension $N[G]$ is a $\kappa$-model in $V[G]$ by the generic closure criterion (Lemma \ref{le:genericclosure}) as $\p^f_\kappa$ has the $\kappa$-cc. It follows that $\Add(\kappa,1)^{N[G]}=\Add(\kappa,1)^{V[G]}$  and $\Add(\kappa,j(f)(\kappa))^{N[G]}$ is isomorphic to $\Add(\kappa,1)^{V[G]}$. To ensure that $g$ is a factor of $\tilde{H}_{(\kappa,\bar{\kappa}^N)}$, we will use it on the first coordinate of $\Add(\kappa,j(f)(\kappa))^{N[G]}$. We obtain a $V[G][g]$-generic filter $h$ for the remaining coordinates from the $V[G][g]$-generic filter for $\Add(\kappa,1)$ on the second coordinate of $K$. This defines $\tilde{H}_\kappa$, an $N[G]$-generic for $\Add(\kappa,j(f)(\kappa))^{N[G]}$. We build the $N[G]$-generic filter $\tilde{H}_{(\kappa,\bar{\kappa}^N)}$ for $\tilde{\E}^{j(f)}_{(\kappa,\bar{\kappa}^N)}$ using the diagonalization criterion (Lemma~\ref{le:diag(1)}) in $V[G]$. As mentioned above, $\tilde{H}_\kappa$ and $\tilde{H}_{(\kappa,\bar{\kappa})}$ are mutually generic and hence $H:=\tilde{H}_\kappa\times\tilde{H}_{(\kappa,\bar{\kappa})}$ is $N[G]$-generic for $\dot{\Add}(\kappa,j(f)(\kappa))\times\dot{\tilde\E}^{j(f)}_{(\kappa,\bar{\kappa}^N)}$.


We already observed that $N[G]$ is a $\kappa$-model in $V[G]$, and it remains so in $V[G][K]$ by the closure of $\Add(\kappa,F(\kappa))^{N[G]}$. It follows that $N[G][H]$ is a $\kappa$-model in $V[G][K]$ by  the ground closure criterion (Lemma \ref{le:groundclosure}). Thus, we use the diagonalization criterion to build an $N[G][H]$-generic filter $\tilde{G}_\tail$ for $\tilde\p_\tail^{j(f)}:=(\dot{\tilde\P}^{j(f)}_{\tail})_{G*H}$. We are now able to lift $j$ to $j:M[G]\to N[j(G)]$ with $j(G)=G*H*\tilde G_\tail$ in $V[G][K]$.

It remains to lift $j$ to $M[G][g]$, for which it suffices, using the lifting criterion,  to construct,  in $V[G][K]$, an $N[j(G)]$-generic filter for the poset $j(\Add(\kappa,1))=\Add(j(\kappa),1)^{N[j(G)]}$ containing $j\image g=g$. Since $N[j(G)]$ is a $\kappa$-model in $V[G][K]$ by the ground closure criterion, we  use the diagonalization criterion to construct an $N[j(G)]$-generic filter $g^*$ for $\Add(j(\kappa),1)^{N[j(G)]}$. To ensure that $g\of g^*$, we carry out the diagonalization below the condition $\Union g\in \Add(j(\kappa),1)^{N[j(G)]}$. We are now able to lift $j$ to $j:M[G][g]\to N[j(G*g)]$ with $j(G*g)=G*H*\tilde G_\tail*g^*$, thus completing the argument.
\end{proof}
Notice that the lift $j:M[G][g]\to N[j(G*g)]$ is clearly not $\kappa$-powerset preserving since the $N[G][g]$-generic filter $h$, which is a factor of $H$, adds new subsets of $\kappa$ not in $M[G][g]$. This was the chief obstacle that had to be overcome to modify the argument for strongly Ramsey and Ramsey cardinals. In these later arguments, we will force over $M[G]$ with $\Add(\kappa,\kappa^+)^{M[G]}$ instead of $\Add(\kappa,1)$. As we shall see below, this will allow us include the necessary master conditions for the second stage of the lift in the $N[G]$-generic filter for $\Add(\kappa,j(f)(\kappa))^{N[G]}$ without introducing subsets of $\kappa$ outside of $M[G][g]$. For $\Add(\kappa,\kappa^+)$ to make sense in $M[G]$, we need to work with models of full $\ZFC$. The strategy was motivated by  Levinski, who used the iteration $\p_\kappa*\Add(\kappa,\kappa^+)$, where $\p_\kappa$ forced with $\Add(\gamma,\gamma^{++})$ at successor stages $\gamma$, to produce a model in which the \GCH\ fails for the first time at a measurable $\kappa$ \cite{levinski:gch}.

In our analysis of the posets involved in the next two theorems we will use the following.
\begin{remark}\label{re:ordinal}
If $\kappa\subseteq X\prec M\models\ZFC^-$, then $X\cap(\kappa^+)^M$ is an ordinal. To demonstrate this, let us suppose $\alpha\in X\cap (\kappa^+)^M$ and argue that if $\beta<\alpha$ then $\beta\in X$. In $M$, there is a bijection from $\kappa$ to $\alpha$. By elementarity, fix $f\in X$ such that $X\models$ ``$f:\kappa\to\alpha$ is a bijection.'' Again by elementarity, $M\models$ ``$f:\kappa\to\alpha$ is a bijection'' and hence there is a $\gamma<\kappa$ such that $M\models$ $f(\gamma)=\beta$. Now, since $\kappa\subseteq X$ it easily follows that $X\models$ ``$\gamma\in\dom(f)$'' and hence $\beta=f(\gamma)\in X$.
\end{remark}

We are now ready to prove the main theorem, which we split here into two theorems, one for strongly Ramsey cardinals and one for Ramsey cardinals. It is worth noting that \GCH\ can be forced over any universe containing a Ramsey or a strongly Ramsey cardinal without destroying these large cardinals \cite{gitman:ramseyindes}.

\begin{theorem}\label{th:stronglyramsey}
Assuming \GCH, if $\kappa$ is a strongly Ramsey cardinal and $F$ is a class function on the regular cardinals having a closure point at $\kappa$ and satisfying $F(\alpha)\leq F(\beta)$ for $\alpha\leq\beta$ and $\alpha<\text{cf}(F(\alpha))$, then there is a confinality preserving forcing extension in which $\kappa$ remains strongly Ramsey and $F$ is realized as the continuum function, that is $2^\delta=F(\delta)$ for every regular cardinal $\delta$.
\end{theorem}
\begin{proof}
As in the proof of Theorem~\ref{th:weaklycompact}, it suffices to argue that $\kappa$ remains strongly Ramsey after forcing with $\p^F_\kappa*\dot\Add(\kappa,F(\kappa))$. Towards this end, let $G*K\subseteq\p_\kappa^F*\dot\Add(\kappa,F(\kappa))$ be $V$-generic, and consider the extension $V[G][K]$. We need to show that every $A\of\kappa$ in $V[G][K]$ is contained there in a $\kappa$-model for which there is a $\kappa$-powerset preserving elementary embedding.

We fix $A\of\kappa$ in $V[G][K]$ and choose a nice
$\p_\kappa^F*\dot\Add(\kappa,F(\kappa))$-name $\dot A$ in $V$ such that $(\dot A)_{G*K}=A$. As in the proof of Theorem \ref{th:weaklycompact}, we assume that $\dot A$ is a nice
$\p_\kappa^F*\dot\Add(\kappa,1)$-name and $(\dot A)_{G*g}=A$, where $g$ is the $V[G]$-generic filter for $\Add(\kappa,1)^{V[G]}$ obtained from the first coordinate of $K$. Using Lemma \ref{le:strongramseycharacterization}, we fix a $\kappa$-special $\kappa$-model $M$ containing $\dot
A$ and $f=F\restrict\kappa$ with the $\kappa$-powerset preserving ultrapower map $j:M\to N$ having $M=V_{j(\kappa)}^N\in N$. Also, we fix a sequence $\la X_\xi\mid\xi<\kappa\ra$ witnessing that $M$ is $\kappa$-special. By throwing away an initial segment of the sequence if necessary, we may assume without loss of generality that $\kappa,f\in X_0$.  Using Lemma \ref{le:specialult}, we likewise fix a sequence $\la Y_\xi\mid\xi<\kappa\ra$ witnessing that $N$ is $\kappa$-special with $\kappa,f\in Y_0$.  Unlike the proof of Theorem \ref{th:weaklycompact}, where we forced over $M$ with the iteration $\p^f_\kappa*\dot\Add(\kappa,1)$, we will force here with $(\p^f_\kappa*\dot\Add(\kappa,\kappa^+))^M$. We shall construct an $M[G]$-generic filter $H$ for the poset $\Add(\kappa,\kappa^+)^{M[G]}$ containing $g$ and lift $j$ to a $\kappa$-powerset preserving embedding of $M[G][H]$, which we will argue is a $\kappa$-model in $V[G][K]$, thus witnessing that $A$ can be put into a $\kappa$-model for which there is a $\kappa$-powerset preserving embedding.

Observe that both $M[G]$ and $N[G]$ are $\kappa$-models in $V[G]$ by the generic closure criterion as $\p_\kappa^f$ has the $\kappa$-cc. Moreover, both are $\kappa$-special as witnessed by the sequences $\la \overline X_\xi\mid\xi<\kappa\ra$ and $\la\overline Y_\xi\mid\xi<\kappa\ra$ respectively, where $\overline X_\xi=X_\xi[G]$ and $\overline Y_\xi=Y_\xi[G]$. To verify this, we first observe that $G$ is $X_{\xi+1}$-generic for every $\xi<\kappa$. The poset $\p^f_\kappa$ is an element of every $X_{\xi+1}$ since it is definable from $f$ and $\kappa$. Furthermore, every $M$-generic filter for $\p^f_\kappa$ is automatically $X_{\xi+1}$-generic since every antichain of $\p^f_\kappa$ that is an element of $X_{\xi+1}$ is also a subset of $X_{\xi+1}$ by the closure of $X_{\xi+1}$ in $M$. It follows that $ X_{\xi+1}[G]^{\ltkappa}\of X_{\xi+1}[G]$ by the generic closure criterion. Using Remark~\ref{rem:xgenericelem}, it also follows that each $X_{\xi+1}[G]\prec M[G]$, but then it is easy to see that for limit $\xi$, we have $X_\xi[G]\prec M[G]$ as well. An identical argument verifies these properties for the $Y_\xi[G]$-sequence.


Since $\Add(\kappa,1)^{M[G]}=\Add(\kappa,1)^{V[G]}$ and $\Add(\kappa,\kappa^+)^{M[G]}$ is isomorphic to\break $\Add(\kappa,1)^{V[G]}$, it follows that we may obtain an $M[G]$-generic filter $H$ for\break $\Add(\kappa,\kappa^+)^{M[G]}$ by using $g$ on the first coordinate and using the second coordinate of $K$ on the remaining coordinates of $\Add(\kappa,\kappa^+)^{M[G]}$. As in the proof of Theorem \ref{th:weaklycompact}, we first lift $j$ to $M[G]$, while ensuring that the master conditions for the second lift are included in $N[j(G)]$. Using the lifting criterion, it suffices to find an $N$-generic filter for the poset $j(\p_\kappa^f)$ containing $j\image G=G$ as a subset. As in (\ref{eqn:factorization}) above, we may factor the iteration $j(\p_\kappa^f)\cong \p_\kappa^f*(\dot\Add(\kappa,j(f)(\kappa))\times\dot{\tilde{\E}}^{j(f)}_{(\kappa,\bar{\kappa}^N)})*\dot{\tilde{\p}}_\tail^{j(f)}$ in $N$. See the discussion around (\ref{eqn:factorization}) for a definition of the notation. We use $G$ as the $N$-generic filter for the $\p_\kappa^f$ part of the forcing to satisfy the lifting criterion.


We shall construct an $N[G]$-generic filter for $\Add(\kappa,j(f)(\kappa))^{N[G]}$ as a particular isomorphic copy of $H$. We will partition the posets $\Add(\kappa,\kappa^+)^{M[G]}$ and $\Add(\kappa,j(f)(\kappa))^{N[G]}$ each into $\kappa$ pieces and argue that the corresponding pieces are isomorphic.
Let $x_0=\overline X_1\cap (\kappa^+)^{M[G]}$ and for $\xi>0$, let $x_\xi=(\overline X_{\xi+1}\setminus \overline X_\xi)\cap (\kappa^+)^{M[G]}$. Observe that each $x_\xi$ has size $\kappa$ in $M[G]$ and $(\kappa^+)^{M[G]}$ is the disjoint union of the $x_\xi$. For $\xi<\kappa$, we let $\q_\xi$ be the subposet of $\Add(\kappa,\kappa^+)^{M[G]}$ consisting of conditions $p$ with domain contained in $x_\xi$. Each such poset $\q_\xi$ is an element of $M[G]$ and all full support products $\Pi_{\xi<\delta}\q_\xi$ for $\delta<\kappa$ are elements of $M[G]$ as well by closure. Moreover, $M[G]$ sees that each $\q_\xi$ is isomorphic to $\Add(\kappa,1)^{M[G]}=\Add(\kappa,1)^{V[G]}$, and so we may fix isomorphisms $\varphi_\xi\in M[G]$ witnessing this. Finally, observe that each full support product $\Pi_{\xi<\delta}\q_\xi$ is naturally isomorphic to $\overline X_\delta\cap \Add(\kappa,\kappa^+)^{M[G]}$. Now, we perform a similar partition of $\Add(\kappa,j(f)(\kappa))^{N[G]}$. We let $y_0=\overline Y_1\cap j(f)(\kappa)$ and $y_\xi=(\overline Y_{\xi+1}\setminus \overline Y_\xi)\cap j(f)(\kappa)$, and define $\rr_\xi$ to be the subposet of $\Add(\kappa,j(f)(\kappa))^{N[G]}$ consisting of conditions $p$ with domain contained in $y_\xi$. We also fix isomorphisms $\psi_\xi\in N[G]$ witnessing that the $\rr_\xi$ is isomorphic to $\Add(\kappa,1)^{N[G]}=\Add(\kappa,1)^{V[G]}$ and note that each full support product $\Pi_{\xi<\delta}\rr_\xi$ is isomorphic to $\overline Y_\delta\cap \Add(\kappa,j(f)(\kappa))^{N[G]}$ . In $V[G]$, the poset $\Add(\kappa,\kappa^+)^{M[G]}$ may be viewed as the bounded support product $\Pi_{\xi<\kappa}\q_\xi$ and the poset $\Add(\kappa,j(f)(\kappa))^{N[G]}$ may be viewed as the bounded-support product $\Pi_{\xi<\kappa}\rr_\xi$. It follows that using the isomorphisms $\varphi_\xi$ and $\psi_\xi$, we may define in $V[G]$, an isomorphism of $\Add(\kappa,\kappa^+)^{M[G]}$ with $\Add(\kappa,j(f)(\kappa))^{N[G]}$ that maps $\q_\xi$ isomorphically onto $\rr_\xi$. Via this isomorphism, we obtain an $N[G]$-generic filter $\tilde H_\kappa$ for $\Add(\kappa,j(f)(\kappa))^{N[G]}$ from $H$.


Since ${\tilde{\E}}^{j(f)}_{(\kappa,\bar{\kappa}^N)}$ is ${\leq}\kappa$-closed in $N[G]$, it follows by the diagonalization criterion that there is an $N[G]$-generic filter $\tilde{H}_{(\kappa,\bar{\kappa}^N)}$ in $V[G]$ for ${\tilde{\E}}^{j(f)}_{(\kappa,\bar{\kappa}^N)}$. Now since $\Add(\kappa,j(f)(\kappa))^{N[G]}$ has the $\kappa^+$-cc and ${\tilde{\E}}^{j(f)}_{(\kappa,\bar{\kappa}^N)}$ is ${\leq}\kappa$-closed in $N[G]$, it follows that the generic filters $\tilde H_\kappa$ and $\tilde{H}_{(\kappa,\bar{\kappa}^N)}$ are mutually generic, and hence $\tilde{H}:=\tilde H_\kappa\times\tilde H_{(\kappa,\bar{\kappa}^N)}$ is $N[G]$-generic for $\Add(\kappa,j(f)(\kappa))\times\tilde \E^{j(f)}_{(\kappa,\bar{\kappa}^N)}$.

The one consequence of our particular choice of $\tilde H_\kappa$ is that we can now argue that for every $\delta<\kappa$, the sets $H\cap \overline X_\delta$ are elements of $N[G][\tilde H]$. Because we used the isomorphisms $\varphi_\xi\in M[G]\subseteq N[G]$ and $\psi_\xi\in N[G]$ to construct $\tilde H_\kappa$ from $H$, we now use the sequences $\la \varphi_\xi\mid\xi<\delta\ra$ and $\la\psi_\xi\mid\xi<\delta\ra$ to reverse the process. We shall use that the increasingly large pieces of $H$, namely $H\cap \overline X_\delta$, are in $N[G][\tilde H]$ to construct the necessary master conditions for the second stage of the lift. Our construction of $\tilde H_\kappa$ will also prove instrumental in showing that the final lift is $\kappa$-powerset preserving.  Observe that any subset $y$ of $\kappa$ in $N[G][\tilde{H}]$ is also in $N[G][\tilde{H}_\kappa]$ by Easton's Lemma (Lemma \ref{le:easton}) and thus has a nice $\Add(\kappa,j(f)(\kappa))^{N[G]}$-name $\dot{y}$ in some $\overline Y_\delta$. It follows that, using the sequences $\la \varphi_\xi\mid\xi<\delta\ra$ and $\la\psi_\xi\mid\xi<\delta\ra$, $N[G]$ can construct a nice $\Add(\kappa,\kappa^+)^{M[G]}$-name $\dot x$ such that $(\dot x)_H=y$. We shall see below that this suffices to argue that the final lift is $\kappa$-powerset preserving.

We now finish the construction of the $N$-generic filter for $j(\P^f_\kappa)$. The model $N[G]$ remains a $\kappa$-model in $V[G][K]$ by the closure of $\Add(\kappa,F(\kappa))^{V[G]}$, and so $N[G][\tilde{H}]$ is a $\kappa$-model as well in $V[G][K]$ by the ground closure criterion.  Thus, we use the diagonalization criterion to construct an $N[G][\tilde{H}]$-generic filter $\tilde{G}_\tail$ for $\tilde\p_\tail^{j(f)}$. We are now able to lift $j$ to $j:M[G]\to N[j(G)]$ with $j(G)=G*\tilde{H}*\tilde{G}_\tail$ in $V[G][K]$.

Next, we complete the lift of the embedding $j$ to $M[G][H]$ by constructing an $N[j(G)]$-generic filter for the poset $j(\Add(\kappa,\kappa^+))=\Add(j(\kappa),j(\kappa)^+)^{N[j(G)]}$ containing $j\image H$ as a subset. First, we argue that $N[j(G)]$ contains increasingly large pieces of $j\image H$; these pieces, $j\image (H\cap \overline X_\xi)$ for each $\xi<\kappa$, will be used to obtain increasingly powerful master conditions for the lift. Let us show that if $D\in M[G]$ and $|D|^{M[G]}=\kappa$ then $j\restrict D \in N[j(G)]$. Fix some bijection $\sigma:\kappa\to D$. Indeed, for $a\in D$, we have $j(a)=j(\sigma(\xi))=j(\sigma)(\xi)$, where $\sigma(\xi)=a$, and thus $j\restrict D$ can be constructed from $\sigma$ and $j(\sigma)$, both of which are in $N[j(G)]$. Recall that we have $H\cap\overline X_\xi \in N[j(G)]$ for every $\xi<\kappa$. From the above remarks, it now follows that since $|\overline X_\xi|^{M[G]}=\kappa$, we have $j\restrict\overline X_\xi\in N[j(G)]$ and thus $j\image(H\cap \overline X_\xi)\in N[j(G)]$. This implies that for each $\xi<\kappa$, the union $q_\xi:=\bigcup j\image(H\cap\overline X_\xi)$ is a condition in $\Add(j(\kappa),j(\kappa)^+)^{N[j(G)]}$. The conditions $q_\xi$ will be the increasingly powerful master conditions.


We construct an $N[j(G)]$-generic filter for $\Add(j(\kappa),j(\kappa)^+)^{N[j(G)]}$ containing $j\image H$ using a
diagonalization argument. Let $\alpha_\xi=\overline X_\xi\cap (\kappa^+)^{M[G]}$, which is an ordinal since $\kappa\of \overline X_\xi$ (see Remark \ref{re:ordinal}), and note that $\text{dom}(q_\xi)=j\image\alpha_\xi\of j(\alpha_\xi)$. The ordinals $\alpha_\xi$ are obviously unbounded in $(\kappa^+)^{M[G]}$ and therefore, since $j$ is an ultrapower map, it follows that the ordinals $j(\alpha_\xi)$ are unbounded in $(j(\kappa)^+)^{N[j(G)]}$. Consider a maximal antichain $A$ of $\Add(j(\kappa),j(\kappa)^+)^{N[j(G)]}$ in $N[j(G)]$ and note that, since $\Add(j(\kappa),j(\kappa)^+)^{N[j(G)]}$ has the $j(\kappa)^+$-cc in $N[j(G)]$, there must be an $\alpha_\xi$ such that $A$ is a maximal antichain of $\Add(j(\kappa),j(\alpha_\xi))^{N[j(G)]}$. In  $V[G][K]$, we enumerate the maximal antichains of $\Add(j(\kappa),j(\kappa)^+)^{N[j(G)]}$ that are elements of $N[j(G)]$ in a sequence $\la A_\xi\mid\xi<\kappa\ra$. We shall build a descending $\kappa$-sequence of conditions in $\Add(j(\kappa),j(\kappa)^+)^{N[j(G)]}$ compatible to $j\image H$, which will be used to generate an $N[j(G)]$-generic filter containing $j\image H$. Fix $\alpha_{\xi_0}$ such that $A_0$ is a maximal antichain of
$\Add(j(\kappa),j(\alpha_{\xi_0}))^{N[j(G)]}$ and find a condition $p_0\in
\Add(j(\kappa),j(\alpha_{\xi_0}))^{N[j(G)]}$ below $q_{\xi_0}$ and some element
of $A_0$.  Since
$\text{dom}(q_{\xi_0})=j\image{\alpha_{\xi_0}}$, the condition $p_0$ is compatible to $j\image H$. Now assume inductively
that we have a descending sequence of conditions $\la p_\zeta\mid
\zeta<\delta\ra$ for some $\delta<\kappa$ compatible to $j\image H$
and each $p_\zeta$ has above it some element of $A_\zeta$. We
build $p_\delta$ as follows. Let
$p^*_\delta=\Union_{\zeta<\delta}p_\zeta$, and fix
$\alpha_{\xi_\delta}$ such that $\text{dom}(p^*_\delta)\subseteq j(\alpha_{\xi_\delta})$ and $A_\delta$ is a maximal antichain of $\Add(j(\kappa),j(\alpha_{\xi_\delta}))^{N[j(G)]}$. Since
$p^*_\delta$ is compatible to $j\image H$ by the inductive assumption, we can find a condition
$p_\delta^{**}$ extending $p^*_\delta$ and $q_{\xi_\delta}$, and
then choose $p_\delta\in \Add(j(\kappa),j(\alpha_{\xi_\delta}))^{N[j(G)]}$ below $p^{**}_\delta$ and below some element of
$A_\delta$. Let $J$ be any filter generated by the sequence $\la p_\xi\mid\xi<\kappa\ra$. It is clear that $J$ is $N[j(G)]$-generic for $\Add(j(\kappa),j(\kappa)^+)^{N[j(G)]}$ and contains $j\image H$.
We are now able to lift $j$ to $j:M[G][H]\to N[j(G*H)]$ with
$j(G*H)=j(G)*J$ in $V[G][K]$.

We observed earlier that the model $M[G]$ is a $\kappa$-model in $V[G]$ and hence also in $V[G][K]$ by the ground closure criterion. Thus, the model $M[G][H]$ is a $\kappa$-model in $V[G][K]$ by another application of the ground closure criterion. To complete the proof, it remains to verify that the lift of $j$ is $\kappa$-powerset preserving. Since $\p_\tail^{j(f)}*\Add(j(\kappa),j(\kappa)^+)$ is $\leqkappa$-closed in $N[G][\tilde{H}]$, it follows that
$N[G][\tilde{H}]$ and $N[j(G*H)]$ have the same subsets of $\kappa$. Thus, it suffices to show that every subset of $\kappa$ in $N[G][\tilde{H}]$ is an element of $M[G][H]$. First, we note that every subset of $\kappa$ in $N[G]$ is an element of $M[G]$ since it has a $\p_\kappa^f$-name in $\her{\kappa}^N\of M$. By our earlier observation, for every $B\of\kappa$ in $N[G][\tilde{H}]$, there is a nice $\Add(\kappa,\kappa^+)^{M[G]}$-name $\dot B\in N[G]$ such that $\dot B_H=B$. Since $\dot B$ can be coded by a subset of $\kappa$,  it follows that $\dot B\in M[G]$, and hence $B\in M[G][H]$.
\end{proof}

\begin{theorem}\label{th:ramsey}
Assuming \GCH, if $\kappa$ is a Ramsey cardinal and $F$ is a class function on the regular cardinals having a closure point at $\kappa$ and satisfying $F(\alpha)\leq F(\beta)$ for $\alpha\leq\beta$ and $\alpha<\text{cf}(F(\alpha))$, then there is a confinality preserving forcing extension in which $\kappa$ remains Ramsey and $F$ is realized as the continuum function, that is $2^\delta=F(\delta)$ for every regular cardinal $\delta$.
\end{theorem}
\begin{proof}
As in the previous arguments, it suffices to show that $\kappa$ remains Ramsey after forcing with $\p^F_\kappa*\dot\Add(\kappa,F(\kappa))$. Towards this end, let $G*K\subseteq\p_\kappa^F*\dot\Add(\kappa,F(\kappa))$ be $V$-generic, and consider the extension $V[G][K]$. We need to show that every $A\of\kappa$ in $V[G][K]$ is contained there in a weak $\kappa$-model $M$ for which there exists a weakly amenable $\omega_1$-intersecting $M$-ultrafilter on $\kappa$.

We fix $A\of\kappa$ in $V[G][K]$ and choose a nice
$\p_\kappa^F*\dot\Add(\kappa,1)$-name $\dot A$ in $V$ such that $(\dot A)_{G*g}=A$, where $g$ is the $V[G]$-generic  for $\Add(\kappa,1)^{V[G]}$ obtained from the first coordinate of $K$. Using Lemma \ref{le:ramseycharacterization}, we fix an $\omega$-special weak $\kappa$-model $M$ containing $V_\kappa$, $\dot
A$ and $f=F\restrict\kappa$ for which there is an $\omega_1$-intersecting $M$-ultrafilter $U$ on $\kappa$ with the $\kappa$-powerset preserving ultrapower map $j:M\to N$ having $M=V_{j(\kappa)}^N\in M$.  Also, we fix a sequence $\la X_n\mid n<\omega\ra$ witnessing that $M$ is $\omega$-special. As before, by throwing away an initial segment of the sequence if necessary, we may assume  that $\kappa,f\in X_0$. Using Lemma \ref{le:specialult}, we fix a sequence $\la Y_n\mid n<\omega\ra$ witnessing that $N$ is $\omega$-special with $\kappa,f\in Y_0$. Observe for the future that the models $M[G]$ and $N[G]$ are $\omega$-special as witnessed by the sequences $\la \overline X_n\mid n<\omega\ra$ and $\la \overline Y_n\mid n<\omega\ra$ respectively, where $\overline X_n=X_n[G]$ and $\overline Y_n=Y_n[G]$. As in the proof of Theorem \ref{th:stronglyramsey}, we will force over $M$  with the iteration $(\p^f_\kappa*\Add(\kappa,\kappa^+))^M$. We shall construct an $M[G]$-generic filter $H$ for the poset $\Add(\kappa,\kappa^+)^{M[G]}$ containing $g$ and lift $j$ to a $\kappa$-powerset preserving embedding of $M[G][H]$. We will argue that the lift of $j$ is the ultrapower by an $\omega_1$-intersecting $M[G][H]$-ultrafilter, thus witnessing that $A$ can be put into a weak $\kappa$-model $M[G][H]$ for which there exists a weakly amenable $\omega_1$-intersecting $M[G][H]$-ultrafilter on $\kappa$.

We cannot construct the filter $H$ as in the proof of Theorem \ref{th:stronglyramsey} since, unlike the situation in that argument, here $\Add(\kappa,\kappa^+)^{M[G]}$ is not isomorphic to $\Add(\kappa,1)^{V[G]}$. Indeed, the poset $\Add(\kappa,\kappa^+)^{M[G]}$ is not even countably closed as is witnessed by any descending sequence $\la p_n\mid n<\omega\ra$ of conditions in $\Add(\kappa,\kappa^+)^{M[G]}$ with $p_n\in X_n$. However,  since $V_\kappa\of M$, it is still true that $\Add(\kappa,1)^{M[G]}=\Add(\kappa,1)^{V[G]}$, and we shall exploit this fact in constructing an $M[G]$-generic filter for $\Add(\kappa,\kappa^+)^{M[G]}$. Let $x_0=\overline X_1\cap (\kappa^+)^{M[G]}$ and for $n>0$, let $x_n=(\overline X_{n+1}\setminus \overline X_n)\cap (\kappa^+)^{M[G]}$. Observe that each $x_n$ has size $\kappa$ in $M[G]$ and $(\kappa^+)^{M[G]}$ is the disjoint union of the $x_n$. For $n<\omega$, we let $\q_n$ be the subposet of $\Add(\kappa,\kappa^+)^{M[G]}$ consisting of conditions $p$ with domain contained in $x_n$. Each such poset $\q_n$ is an element of $M[G]$ and thus all products $\Pi_{n<m}\q_n$ for $m<\omega$ are elements of $M[G]$ as well. Moreover, $M[G]$ sees that each $\q_n$ is isomorphic to $\Add(\kappa,1)^{M[G]}=\Add(\kappa,1)^{V[G]}$, and so we may fix isomorphisms $\varphi_n\in M[G]$ witnessing this. Finally, observe that a product $\Pi_{n<m}\q_n$ is naturally isomorphic to $\overline X_m\cap \Add(\kappa,\kappa^+)^{M[G]}$. In $V[G]$, we may view the poset $\Add(\kappa,\kappa^+)^{M[G]}$ as the finite-support product $\Pi_{n<\omega}\q_n$. It follows that using the isomorphisms $\varphi_n$, we may define in $V[G]$, an isomorphism of $\Add(\kappa,\kappa^+)^{M[G]}$ with the finite-support product of $\omega$-many copies of $\Add(\kappa,1)$, that maps $\q_n$ isomorphically onto the $n^{\text{th}}$-copy $\Add(\kappa,1)$. Let $H'$ be the $V[G]$-generic filter for $\Add(\kappa,\omega)$ obtained from the first $\omega$-many coordinates of $K$, and note that $g$ is the filter on the first coordinate of $H'$. Let $H$ be the subset of $H'$ consisting of conditions with finite support. Since the restriction of a generic filter for a full support product to the finite support product is never generic, the filter $H$ will not be $V[G]$-generic for the finite-support product $\Pi_{n<\omega}\q_n$. However, we shall argue that it is, in fact, $M[G]$-generic, which will suffice. Observe that, since $\Add(\kappa,\kappa^+)^{M[G]}$ has the $\kappa^+$-cc in $M[G]$, every antichain of it that is an element of $M[G]$ must be contained in some $\Pi_{n<m}\q_n$. The restriction of $H$ to every product $\Pi_{n<m}\q_n$ is obviously $V[G]$-generic, and so it follows that it must meet every antichain of $\Add(\kappa,\kappa^+)^{M[G]}$ that is in $M[G]$. This completes the argument that $H$ is an $M[G]$-generic filter for $\Add(\kappa,\kappa^+)^{M[G]}$. Next, we will lift $j$ to $M[G][H]$.

As in the previous proofs, we first lift $j$ to $M[G]$, while ensuring that the master conditions for the second lift are included in $N[j(G)]$. Using the lifting criterion, it suffices to find an $N$-generic filter for the poset $j(\p_\kappa^f)$ containing $j\image G=G$ as a subset. As before we factor the iteration $j(\P^f_\kappa)$ as $j(\p_\kappa^f)\cong \p_\kappa^f*(\dot\Add(\kappa,j(f)(\kappa))\times\dot{\tilde{\E}}^{j(f)}_{(\kappa,\bar{\kappa}^N)})*\dot{\tilde{\p}}_\tail^{j(f)}$ in $N$. See the discussion around (\ref{eqn:factorization}) for a definition of the notation. We use $G$ as the $N$-generic filter for the $\p_\kappa^f$ part of the forcing to satisfy the lifting criterion. As in the proof of Theorem \ref{th:stronglyramsey}, we shall construct an $N[G]$-generic filter for $\Add(\kappa,j(f)(\kappa))^{N[G]}$ as a particular isomorphic copy of $H$. We let $y_0=\overline Y_1\cap j(f)(\kappa)$ and $y_n=(\overline Y_{n+1}\setminus \overline Y_n)\cap j(f)(\kappa)$, and define $\rr_n$ to be the subposet of $\Add(\kappa,j(f)(\kappa))^{N[G]}$ consisting of conditions $p$ with domain contained in $y_n$. We also fix isomorphisms $\psi_n\in N[G]$ witnessing that the $\rr_n$ is isomorphic to $\Add(\kappa,1)^{N[G]}=\Add(\kappa,1)^{V[G]}$, and note that each product $\Pi_{n<m}\rr_n$ is isomorphic to $\overline Y_m\cap \Add(\kappa,j(f)(\kappa))^{N[G]}$. In $V[G]$, we may view the poset $\Add(\kappa,j(f)(\kappa))^{N[G]}$ as the finite-support product $\Pi_{n<\omega}\rr_n$. It follows that using the isomorphisms $\varphi_n$ and $\psi_n$, we may define in $V[G]$, an isomorphism of $\Add(\kappa,\kappa^+)^{M[G]}$ with $\Add(\kappa,j(f)(\kappa))^{N[G]}$ that maps $\q_n$ isomorphically onto $\rr_n$. Via this isomorphism, we obtain an $N[G]$-generic filter $\tilde H_\kappa$ for $\Add(\kappa,j(f)(\kappa))^{N[G]}$ from $H$. Since $\tilde{\E}^{j(f)}_{(\kappa,\bar{\kappa}^N)}$ is ${\leq}\kappa$-closed in $N[G]$, it follows from diagonalization criterion~(2) (Lemma \ref{le:diag(2)}) that we may construct an $N[G]$-generic filter $\tilde{H}_{(\kappa,\bar{\kappa}^N)}$ for $\tilde{\E}^{j(f)}_{(\kappa,\bar{\kappa}^N)}$ in $V[G]$. As before, since $\Add(\kappa,j(f)(\kappa))^{N[G]}$ has the $\kappa^+$-cc and $\tilde{\E}^{j(f)}_{(\kappa,\bar{\kappa}^N)}$ is ${\leq}\kappa$-closed in $N[G]$, it follows that $\tilde{H}_\kappa$ and $\tilde{H}_{(\kappa,\bar{\kappa}^N)}$ are mutually generic and thus $\tilde{H}=\tilde{H}_\kappa\times\tilde{H}_{(\kappa,\bar{\kappa}^N)}$ is $N[G]$-generic for $\Add(\kappa,j(f)(\kappa))\times{\tilde{\E}}^{j(f)}_{(\kappa,\bar{\kappa}^N)}$. Now we use diagonalization criterion~(2) to construct an $N[G][\tilde H]$-generic filter $\tilde G_\tail$ for $\p_\tail^{j(f)}$. Thus, we are able to lift $j$ to $j:M[G]\to N[j(G)]$ with $j(G)=G*\tilde H*\tilde G_\tail$ in $V[G][K]$.

As in the proof of Theorem~\ref{th:stronglyramsey}, we can now argue that for every $m<\omega$, the sets $H\cap \overline X_m$ are elements of $N[G][\tilde H]$. Because we used the isomorphisms $\varphi_n\in M[G]\subseteq N[G]$ and $\psi_n\in N[G]$ to construct $\tilde H_\kappa$ from $H$, we use the sequences $\la \varphi_n\mid n<m\ra$ and $\la\psi_n\mid n<m\ra$ to reverse the process. Also, we can make the same argument as in the proof of Theorem~\ref{th:stronglyramsey} that every subset $y$ of $\kappa$ in $N[G][\tilde{H}]$ has a nice $\Add(\kappa,\kappa^+)^{M[G]}$-name $\dot x$ such that $(\dot x)_H=y$.

Next, we complete the lift of the embedding $j$ to $M[G][H]$ by constructing an $N[j(G)]$-generic filter for the poset $j(\Add(\kappa,\kappa^+))=\Add(j(\kappa),j(\kappa)^+)^{N[j(G)]}$ containing $j\image H$ as a subset. As in the proof of Theorem \ref{th:stronglyramsey}, we observe that for each $n<\omega$, the set $j\image (H\cap \overline X_n)$ is an element of $N[j(G)]$, and so
$q_n=\bigcup j\image (H\cap \overline X_n)$ is an element of
$\Add(j(\kappa),j(\kappa)^+)^{N[j(G)]}$. The conditions $q_n$, for $n<\omega$, will be the increasingly powerful master conditions for the lift.

We construct an $N[j(G)]$-generic filter for $\Add(j(\kappa),j(\kappa)^+)^{N[j(G)]}$ containing $j\image H$ using a
diagonalization argument similar to the one in the proof of diagonalization criterion~(2). Let $\alpha_n=\overline X_n\cap (\kappa^+)^{M[G]}$, which is an ordinal since $\kappa\of \overline X_n$, and note that $\text{dom}(q_n)=j\image\alpha_n\of j(\alpha_n)$. The ordinals $\alpha_n$ are obviously unbounded in $(\kappa^+)^{M[G]}$ and therefore, since $j$ is an ultrapower map, the ordinals $j(\alpha_n)$ are unbounded in $(j(\kappa)^+)^{N[j(G)]}$.  For $n<\omega$, let $Z_n=Y_n[j(G)]$, and fix $m_0$ such that $q_0\in Z_{m_0}$. In $N[j(G)]$, using $\leqkappa$-closure of $\Add(j(\kappa),j(\kappa)^+)^{N[j(G)]}$, we construct below $q_0$ a descending $\kappa$-sequence of conditions meeting the $\kappa$-many maximal
antichains of $\Add(j(\kappa),j(\alpha_0))^{N[j(G)]}$ in $Z_{m_0}$
 and choose a condition $p_0$ below the sequence. Since
$\text{dom}(q_0)=j\image{\alpha_0}$, the condition $p_0$ is compatible to $j\image H$. Now assume inductively
that we have a descending sequence of conditions $\la p_n\mid n\leq k\ra$ compatible to $j\image H$,
and each $p_n$ has above it some element of every maximal antichain of
$\Add(j(\kappa),j(\alpha_n))^{N[j(G)]}$ in $Z_{m_n}$. Fix $m_{k+1}$ such that $p_k$ and
$q_{k+1}$ are elements of $Z_{m_{k+1}}$ and choose in
$Z_{m_{k+1}}$, some $p^*_{k+1}$ below both of them, which is possible since $p_k$ is compatible to $j\image H$. Next, in
$N[j(G)]$, we construct below $p^*_{k+1}$  a descending $\kappa$-sequence of conditions meeting  the $\kappa$-many maximal antichains of
$\Add(j(\kappa),j(\alpha_{k+1}))^{N[j(G)]}$ in $Z_{m_{k+1}}$ and choose $p_{k+1}$
below the sequence. In this manner, we obtain a descending
sequence $\la p_n\mid n<\omega\ra$ that is compatible to $j\image H$. We claim that any filter extending $\la p_n\mid n<\omega\ra$ meets all
maximal antichains of $\Add(j(\kappa),j(\kappa)^+)^{N[j(G)]}$ that are elements of $N[j(G)]$. If
$A\in N[j(G)]$ is a maximal antichain of $\Add(j(\kappa),j(\kappa)^+)^{N[j(G)]}$, then
there must be some $\alpha_n$ such that $A\of
\Add(j(\kappa),j(\alpha_n))^{N[j(G)]}$ since $\Add(j(\kappa),j(\kappa)^+)^{N[j(G)]}$ has
$j(\kappa)^+$-cc and the $j(\alpha_n)$ are unbounded in $j(\kappa)^+$.
Also, $A$ clearly remains a maximal antichain in
$\Add(j(\kappa),j(\alpha_m))^{N[j(G)]}$ for all $m\geq n$. Now,  since there must be $k\geq n$ such that
$A\in Z_{m_k}$, the condition $p_k$ must have some
element of $A$ above it by construction.  It follows that any filter containing $p_k$ must meet $A$. Let $J$ be any filter generated by the sequence $\la p_n\mid n<\omega\ra$. It is clear that $J$ contains $j\image H$, and thus, we are able to lift $j$ to $j:M[G][H]\to N[j(G*H)]$ with
$j(G*H)=j(G)*J$ in $V[G][K]$.

The argument that the lift of $j$ is $\kappa$-powerset preserving is identical to the proof of Theorem \ref{th:stronglyramsey}, and so to complete the proof it remains to verify that the lift is the ultrapower by an $\omega_1$-intersecting $M[G][H]$-ultrafilter on $\kappa$. Since $j$ is a lift of an ultrapower embedding, it must itself be an ultrapower embedding by our remarks from Section~\S\ref{se:forcingpreliminaries}, and so we let $W$ be the corresponding $M[G][H]$-ultrafilter. We will argue that $W$ is $\omega_1$-intersecting in two steps.

First, we show that $W\cap M[G]$ is $\omega_1$-intersecting for sequences in $V[G]$.  Since a set $B\of\kappa$ is in $W$ if and only if $\kappa\in j(B)$, it suffices to argue that whenever $\la A_n\mid n<\omega\ra\in V[G]$ is a sequence of subsets of $\kappa$ such that $A_n\in M[G]$ and $\kappa\in j(A_n)$, then $\bigcap_{n<\omega}A_n\neq\emptyset$. Note that only the first poset in the iteration $\p_\kappa^f$ may not be countably closed, and the size of this poset is less than $\kappa$. Since Ramsey cardinals are preserved by small forcing (see Section~\S10 of \cite{kanamori:higher}), we may assume without loss of generality that the poset $\p^f_\kappa$ is countably closed. We fix $\p_\kappa^f$-names $\dot{A}_n\in M$ with $(\dot{A}_n)_G=A_n$, and note
that the sequence $\la \dot{A}_n\mid n\in\omega\ra$ is in $V$ by countable closure of
$\p_\kappa^f$. Next, we fix a $\p_\kappa^f$-name $\dot s$ such that $\one\forces``\dot s$ is an $\omega$-sequence"  and for
all $n<\omega$, $\one \forces\dot s(\check{n})=\dot{A}_n$
over $V$ (such a name is constructed from the sequence $\la
\dot{A}_n\mid n\in\omega\ra$). Now suppose to the contrary that
$\bigcap_{n\in\omega}A_n=\emptyset$, and choose $p\in G$ such
that $p\forces \bigcap\dot s=\emptyset$ over $V$.
Since $p\in G$, we have $j(p)=p\in j(G)$. Using that $j(G)$ is a filter, we find a descending sequence of conditions below $p$ in $j(G)$, $p\geq p_1\geq p_2\geq\cdots\geq p_n\cdots$,  such
that $p_n\forces \check{\kappa}\in j(\dot{A}_n)$ over $N$. Since $j$ is the ultrapower by an $M$-ultrafilter $U$, we may fix for each $n<\omega$, a function $f_n:\kappa\to
\p_\kappa^f$ in $M$ such that $p_n=[f_n]_U$. Note that, by closure of $\p_\kappa^f$,  the
sequence $\la f_n\mid n\in\omega\ra$ is in $V$. Now observe that the
following sets are in $U$:
\begin{itemize}
\item[(1)] $S_n=\{\xi<\kappa\mid f_n(\xi)\forces \check{\xi}\in
\dot{A}_n\text{ over }M\}$ for $n\in\omega$,
\item[(2)] $T_n=\{\xi<\kappa\mid
f_{n+1}(\xi)\leq f_n(\xi)\}$ for $n\in\omega$,
\item[(3)]
$S=\{\xi<\kappa\mid f_0(\xi)\leq p\}$.
\end{itemize}
Note that the sequences $\la S_n\mid n<\omega\ra$ and $\la T_n\mid n<\omega\ra$ are themselves elements of $V$ by countable closure of $\p$. Thus, since $U$ is $\omega_1$-intersecting in $V$, we may intersect all these sets to obtain an ordinal $\alpha<\kappa$ such that:
\begin{itemize}
\item[(1)] for all $n<\omega$, $f_n(\alpha)\forces \check{\alpha}\in \dot{A}_n$ over $M$,
\item[(2)] for all $n<\omega$, $f_{n+1}(\alpha)\leq f_n(\alpha)$,
\item[(3)] $f_0(\alpha)\leq p$.
\end{itemize}
Once again by closure of $\p_\kappa^f$, we may fix a condition $q$ below the descending sequence $p\geq f_0(\alpha)\geq f_1(\alpha)\geq\cdots\geq f_n(\alpha)\geq\cdots$, which consequently has the following
properties:
\begin{itemize}
\item[(1)] for all $n<\omega$, $q\forces \check{\alpha}\in \dot{A}_n$ over $M$,
\item[(2)] $q\forces \bigcap \dot s=\emptyset$ over $V$.
\end{itemize}
Suppose $\overline G\subseteq \p$ is any $V$-generic filter containing $q$. Since for all $n\in\omega$, $q\forces \check{\alpha}\in \dot{A}_n$ over $M$, we have that $\alpha\in (\dot{A}_n)_{\overline G}$ for all $n\in\omega$ in $V[\overline G]$. On the other hand, since $q\forces \bigcap \dot s=\emptyset$ over $V$ and for all $n\in\omega$, $\one\forces \dot s(\check{n})=\dot{A}_n$ over $V$, we have that $\bigcap
(\dot s)_{\overline G}=\bigcap_{n\in\omega}(\dot{A}_n)_{\overline G}= \emptyset$ in $V[\overline G]$.
Thus, we have reached a contradiction showing that $W$ is $\omega_1$-intersecting for sequences in $V[G]$.

Now, we are ready to show that $W$ is $\omega_1$-intersecting in $V[G][K]$. Towards this end, we fix a sequence $\la A_n\mid n<\omega\ra\in V[G][K]$  of subsets of $\kappa$ such that $A_n\in M[G][H]$ and $\kappa\in j(A_n)$. Using the isomorphisms discussed earlier, we shall view $\Add(\kappa,\kappa^+)^{M[G]}$ as a subposet of $\Add(\kappa,F(\kappa))^{V[G]}$ consisting of conditions whose domains are finite subsets of $\omega$. Recall that this was how we obtained the $M[G]$-generic filter $H$ for $\Add(\kappa,\kappa^+)^{M[G]}$ from the $V[G]$-generic filter $K$ for $\Add(\kappa,F(\kappa))^{V[G]}$. We fix
$\Add(\kappa,\kappa^+)^{M[G]}$-names $\dot{A}_n\in M[G]$ with $(\dot{A}_n)_H=A_n$, and note
that the sequence $\la \dot{A}_n\mid n\in\omega\ra$ is in $V[G]$ by countable closure of
$\Add(\kappa,F(\kappa))^{V[G]}$. Using our assumption above, we may view the names $\dot A_n$ as $\Add(\kappa,F(\kappa))^{V[G]}$-names as well.  Thus, we may fix an $\Add(\kappa,F(\kappa)^{V[G]})$-name $\dot s$ in $V[G]$ such that $\one\forces``\dot s$ is an $\omega$-sequence"  and for
all $n<\omega$, $\one \forces\dot s(\check{n})=\dot{A}_n$. Now suppose to the contrary that
$\bigcap_{n\in\omega}A_n=\emptyset$, and choose $p\in K$ such
that $p\forces \bigcap\dot s=\emptyset$ over $V[G]$. Let $p_n$ be the restriction of $p$ to the first $n$-many coordinates of $\Add(\kappa,F(\kappa))^{V[G]}$, and note that, by our construction, the $M[G]$-generic filter $H$ contains all the conditions $p_n$. It follows that all the conditions $j(p_n)$ are elements of $j(H)$. Using that $j(H)$ is a filter, we may construct a descending sequence $r_0\geq r_1\geq r_2\geq\cdots\geq r_n\geq\cdots$ such that $r_n\forces \check{\kappa}\in j(\dot A_n)$ over $N[j(G)]$ and $r_n\leq j(p_n)$. Since $j:M[G]\to N[j(G)]$ is the ultrapower by an $M[G]$-ultrafilter $W\cap M[G]$, we may fix for each $n<\omega$, a function $f_n:\kappa\to
\Add(\kappa,\kappa^+)^{M[G]}$ in $M[G]$ such that $r_n=[f_n]_{W}$. Note that  the
sequence $\la f_n\mid n\in\omega\ra$ is in $V[G]$ by countable closure of $\Add(\kappa,F(\kappa))^{V[G]}$. Now observe that the
following sets are elements of $W\cap M[G]$:
\begin{itemize}
\item[(1)] $S_n=\{\xi<\kappa\mid f_n(\xi)\forces \check{\xi}\in
\dot{A}_n\text{ over }M[G]\}$ for $n\in\omega$,
\item[(2)] $T_n=\{\xi<\kappa\mid
f_{n+1}(\xi)\leq f_n(\xi)\}$ for $n\in\omega$,
\item[(3)]
$R_n=\{\xi<\kappa\mid f_n(\xi)\leq p_n\}$ for $n\in\omega$.
\end{itemize}
The sequences $\la S_n\mid n<\omega\ra$, $\la T_n\mid
n<\omega\ra$, and $\la R_n\mid n<\omega\ra$ are elements
of $V[G]$ by countable closure of $\Add(\kappa,F(\kappa))^{V[G]}$. Since we showed previously that $W$ is
$\omega_1$-intersecting for sequences in $V[G]$, we may intersect all these sets to obtain an ordinal
$\alpha<\kappa$ such that:
\begin{itemize}
\item[(1)] for all $n<\omega$, $f_n(\alpha)\forces \check{\alpha}\in \dot{A}_n$ over $M[G]$,
\item[(2)] for all $n<\omega$, $f_{n+1}(\alpha)\leq f_n(\alpha)$,
\item[(3)] for all $n<\omega$, $f_n(\alpha)\leq p_n$.
\end{itemize}
Once again, using countable closure of $\Add(\kappa,F(\kappa))^{V[G]}$, we fix a condition $q^*\in \Add(\kappa,F(\kappa))^{V[G]}$ below the descending sequence $f_0(\alpha)\geq f_1(\alpha)\geq f_2(\alpha)\geq \cdots\geq  f_n(\alpha)\geq \cdots$, and note that it is compatible to $p$. Thus, we may fix a condition $q\in \Add(\kappa,F(\kappa))^{V[G]}$ below both $p$ and $q^*$.
If $\overline K$ is any $V[G]$-generic containing $q$, then $\bigcap (\dot s)_{\overline K}=\emptyset$ as $q\leq p$. However, the restriction of $\overline K$ to the $M[G]$-generic $\overline H$ for $\Add(\kappa,\kappa^+)^{M[G]}$ contains all the conditions $f_n(\alpha)$ and hence $\alpha\in (\dot A_n)_{\overline H}=(\dot A_n)_{\overline K}$. Thus, we have reached a contradiction showing that $W$ is fully $\omega_1$-intersecting.
\end{proof}


\begin{thebibliography}{Cum10}

\bibitem[Apt]{apter:choicelesssupercompact}
Arthur Apter.
\newblock On some questions concerning strong compactness.
\newblock In preparation.

\bibitem[Cod12]{cody:dissertation}
Brent Cody.
\newblock {\em Some results on large cardinals and the continuum function}.
\newblock PhD thesis, The Graduate Center of the City University of New York,
  2012.

\bibitem[Cum10]{cummings:weaklycompact}
James Cummings.
\newblock Iterated forcing and elementary embeddings.
\newblock In {\em Handbook of set theory. {V}ols. 1, 2, 3}, pages 775--883.
  Springer, Dordrecht, 2010.

\bibitem[Dod82]{dodd:coremodel}
Anthony Dodd.
\newblock {\em The core model}, volume~61 of {\em London Mathematical Society
  Lecture Note Series}.
\newblock Cambridge University Press, Cambridge, 1982.

\bibitem[Eas70]{easton:gch}
William~Bigelow Easton.
\newblock Powers of regular cardinals.
\newblock {\em Ann. Math. Logic}, 1:139--178, 1970.

\bibitem[EH62]{erdos:ramsey}
P.~Erd{\H{o}}s and A.~Hajnal.
\newblock Some remarks concerning our paper ``{O}n the structure of
  set-mappings''. {N}on-existence of a two-valued {$\sigma $}-measure for the
  first uncountable inaccessible cardinal.
\newblock {\em Acta Math. Acad. Sci. Hungar.}, 13:223--226, 1962.

\bibitem[FH08]{syfriedman:continuum}
Sy-David Friedman and Radek Honzik.
\newblock Easton's theorem and large cardinals.
\newblock {\em Ann. Pure Appl. Logic}, 154(3):191--208, 2008.

\bibitem[Gai74]{gaifman:ultrapowers}
Haim Gaifman.
\newblock Elementary embeddings of models of set-theory and certain
  subtheories.
\newblock In {\em Axiomatic set theory ({P}roc. {S}ympos. {P}ure {M}ath.,
  {V}ol. {XIII}, {P}art {II}, {U}niv. {C}alifornia, {L}os {A}ngeles, {C}alif.,
  1967)}, pages 33--101. Amer. Math. Soc., Providence R.I., 1974.

\bibitem[GHJ]{zfcminus:gitmanhamkinsjohnstone}
Victoria Gitman, Joel~David Hamkins, and Thomas~A. Johnstone.
\newblock What is the theory {ZFC} without power set?
\newblock Submitted.

\bibitem[Git93]{gitik:measurablenotCH}
Moti Gitik.
\newblock On measurable cardinals violating the continuum hypothesis.
\newblock {\em Ann. Pure Appl. Logic}, 63(3):227--240, 1993.

\bibitem[Git11]{gitman:ramsey}
Victoria Gitman.
\newblock {R}amsey-like cardinals.
\newblock {\em The Journal of Symbolic Logic}, 76(2):519--540, 2011.

\bibitem[GJ]{gitman:ramseyindes}
Victoria Gitman and Thomas~A. Johnstone.
\newblock Indestructibility for {R}amsey and {R}amsey-like cardinals.
\newblock In preparation.

\bibitem[GW11]{gitman:welch}
Victoria Gitman and Philip~D. Welch.
\newblock Ramsey-like cardinals {II}.
\newblock {\em The Journal of Symbolic Logic}, 76(2):541--560, 2011.

\bibitem[Jec03]{jech:settheory}
Thomas Jech.
\newblock {\em Set theory}.
\newblock Springer Monographs in Mathematics. Springer-Verlag, Berlin, 2003.
\newblock The third millennium edition, revised and expanded.

\bibitem[Kan09]{kanamori:higher}
Akihiro Kanamori.
\newblock {\em The higher infinite}.
\newblock Springer Monographs in Mathematics. Springer-Verlag, Berlin, second
  edition, 2009.
\newblock Large cardinals in set theory from their beginnings, Paperback
  reprint of the 2003 edition.

\bibitem[Kun70]{kunen:ultrapowers}
Kenneth Kunen.
\newblock Some applications of iterated ultrapowers in set theory.
\newblock {\em Ann. Math. Logic}, 1:179--227, 1970.

\bibitem[Lev95]{levinski:gch}
Jean-Piere Levinski.
\newblock Filters and large cardinals.
\newblock {\em Ann. Pure Appl. Logic}, 72(2):177--212, 1995.

\bibitem[Men76]{menas:indes}
Telis~K. Menas.
\newblock Consistency results concerning supercompactness.
\newblock {\em Trans. Amer. Math. Soc.}, 223:61--91, 1976.

\bibitem[Mit79]{mitchell:ramsey}
William Mitchell.
\newblock Ramsey cardinals and constructibility.
\newblock {\em J. Symbolic Logic}, 44(2):260--266, 1979.

\bibitem[Moo11]{moore:gchhistory}
Gregory Moore.
\newblock Early history of the generalized continuum hypothesis: 1878--1938.
\newblock {\em Bull. Symbolic Logic}, 17(4):489--532, 2011.

\bibitem[Sil75]{silver:singular}
Jack Silver.
\newblock On the singular cardinals problem.
\newblock In {\em Proceedings of the {I}nternational {C}ongress of
  {M}athematicians ({V}ancouver, {B}. {C}., 1974), {V}ol. 1}, pages 265--268.
  Canad. Math. Congress, Montreal, Que., 1975.

\bibitem[Zar82]{zarach:unions_of_zfminus_models}
Andrzej Zarach.
\newblock Unions of {${\rm ZF}^{-}$}-models which are themselves {${\rm
  ZF}^{-}$}-models.
\newblock In {\em Logic {C}olloquium '80 ({P}rague, 1980)}, volume 108 of {\em
  Stud. Logic Foundations Math.}, pages 315--342. North-Holland, Amsterdam,
  1982.

\bibitem[Zar96]{Zarach1996:ReplacmentDoesNotImplyCollection}
Andrzej~M. Zarach.
\newblock Replacement {$\nrightarrow$} collection.
\newblock In {\em G\"odel '96 ({B}rno, 1996)}, volume~6 of {\em Lecture Notes
  Logic}, pages 307--322. Springer, Berlin, 1996.

\end{thebibliography}

\end{document}